\newif\ifsiam
\ifsiam
\documentclass{siamart171218}
\newtheorem{remark}{Remark}[section]
\newtheorem{example}{Example}[section]
\else
\documentclass[12pt]{article}
\usepackage[left=20mm,right=20mm,top=25mm,bottom=25mm]{geometry}

\usepackage{amsthm}

\newtheorem{definition}{Definition}[section]
\newtheorem{remark}{Remark}[section]

\newtheorem{theorem}{Theorem}[section]
\newtheorem{lemma}[theorem]{Lemma}
\newtheorem{proposition}[theorem]{Proposition}

\fi
\usepackage[hidelinks]{hyperref}
\usepackage{amsfonts,amsmath}
\usepackage{upgreek}
\usepackage{graphicx}
\usepackage{bm}
\usepackage{euscript}
\usepackage{pgfplots}
\usepackage{multirow}

\newcommand{\SkM}{S_k^M}
\newcommand{\uhkM}{u_{hk}^M}
\newcommand{\VhkM}{V_{hk}^M}
\newcommand{\bkappa}{\bm{\kappa}}

\renewcommand{\tilde}{\widetilde}

\newcommand{\norm}[2]{\ensuremath{\left\| #1 \right\|_{#2}}}
\newcommand{\abs}[1]{\ensuremath{\left| #1 \right|}}
\newcommand{\seq}[1]{\ensuremath{\left\{ #1 \right\}}}
\newcommand{\sep}[1]{\ensuremath{\left( #1 \right)}}
\newcommand{\eus}[1]{\ensuremath{\EuScript{#1}}}
\newcommand{\Refx}[1]{\mbox{\rm(\ref{#1})}}
\def\NN{\mathbb{N}}\def\RR{\mathbb{R}}\def\PP{\mathbb{P}}\def\I{\mathbb{I}}
\def \A{\eus{A}}\def \B{\eus{B}}\def \F{\eus{F}}\def \R{\eus{R}}
\def \u{{\bf u}}\def \f{{\bf f}}\def \b{{\bf b}}\def \bv{{\bf v}}\def \w{{\bf w}}\def \z{{\bf z}}\def \d{{\bf d}}
\def\x{{\bm{x}}}\def \y{{\bm{y}}}
\def\Nx{N_\x}\def\Ny{N_\y}
\def \zr{{\bf 0}}

\def \balpha{{\bm{\alpha}}}
\def \bGamma{{\bm{\Gamma}}}%works with package amsbsy
\def\dd{{\rm d}} % roman d for use in derivative and integral
\def\card{{\rm card\,}}
\def\fl{{\eus{F}\!\ell}}\def \nnz{\textup{nnz}}
\def\ttau{\tilde\tau}

\DeclareMathOperator{\diag}{diag}
\DeclareMathOperator{\spn}{span}
\DeclareMathOperator{\spp}{supp}

\def\abrev#1{\textcolor{black}{#1}}
\def\rev#1{\textcolor{black}{#1}}

\colorlet{RED}{red}

\begin{document}

\title{Truncation preconditioners for stochastic Galerkin finite element discretizations}
\author{Alex Bespalov\footnotemark[1] \and Daniel Loghin\footnotemark[1] \and Rawin Youngnoi\footnotemark[1]}

%% defines a symbol for one footnote
\long\def\symbolfootnote[#1]#2{\begingroup
\def\thefootnote{\fnsymbol{footnote}}\footnote[#1]{#2}\endgroup}

\renewcommand{\thefootnote}{\fnsymbol{footnote}}
\footnotetext[1]{\,School of Mathematics, University of Birmingham,
Edgbaston, Birmingham B15 2TT, UK ({\tt a.bespalov@bham.ac.uk}, {\tt d.loghin@bham.ac.uk}, {\tt rxy618@student.bham.ac.uk}).\\
{\bf Acknowledgments.} The work of the first author was supported by the EPSRC under grant EP/P013791/1
and by The Alan Turing Institute under the EPSRC grant EP/N510129/1.
}

\renewcommand{\thefootnote}{\arabic{footnote}}

%\keywords{\textcolor{red}{stochastic Galerkin finite element method, two-term preconditioner}}
\ifsiam
\maketitle
\else
\maketitle 
%\tableofcontents
\fi

\begin{abstract}
Stochastic Galerkin finite element method (SGFEM) provides an efficient alternative to traditional sampling methods
for the numerical solution of linear elliptic partial differential equations with parametric or random inputs.
However, computing stochastic Galerkin approximations for a given problem requires
the solution of large coupled systems of linear equations.
Therefore, an effective and bespoke iterative solver is a key ingredient of any SGFEM implementation.
In this paper, \rev{we analyze a class of \emph{truncation preconditioners}} for SGFEM.
%that we call \emph{truncation preconditioners}.
Extending the idea of the mean-based preconditioner, \rev{these} preconditioners capture
additional significant components of the stochastic Galerkin matrix.
Focusing on the parametric diffusion equation as a model problem
and assuming affine-parametric representation of the diffusion coefficient,
we perform spectral analysis of the preconditioned matrices and
establish optimality of truncation preconditioners with respect to SGFEM discretization parameters.
Furthermore, we report the results of numerical experiments for model diffusion problems
with affine and non-affine parametric representations of the coefficient.
In particular, we look at the efficiency of the solver
(in terms of iteration counts for solving the underlying linear systems) and
compare \rev{truncation} preconditioners with other existing preconditioners for stochastic Galerkin matrices,
such as the mean-based %preconditioner
%from [C. E. Powell and H. Elman, IMA J. Numer. Anal., 29: 350?375, 2009]
and the Kronecker product~ones. %preconditioner.
%from [E. Ullmann, SIAM J. Sci. Comput., 32(2): 923-946, 2010].
\end{abstract}

\medskip


\medskip


%-------------------------------------------------------------------
%\newpage
%-----------------------------------------------------------
\section{Introduction}
\label{sec:intro}
%Over the last two decades, new challenges in the field of computational
%partial differential equations (PDE) have been generated by the rapidly
%developing field of uncertainty quantification. An area of great
%interest is developing PDE models with random input data, for which
%several numerical methods have been developed and analysed. In
%particular, the Stochastic Galerkin finite element method (SGFEM --
%see \cite{MR3202242,lord14} ) has
%become an attractive discretisation method in addition to the classical Monte Carlo and
%quasi-Monte Carlo methods (see e.g., \cite{MR2824857}) . However, its
%key drawback is the need to solve huge-scale linear systems. For
%realistic applications, this can only be considered
%iteratively.

\abrev{Over the last two decades,
many new challenging problems in the field of computational partial differential equations (PDEs) have been motivated by
the rapidly developing area of uncertainty quantification.
Efficient numerical solution of PDE problems with parametric or uncertain inputs
is one of these challenges.
Several numerical methods have been developed and analyzed in this context.
In particular, the stochastic Galerkin finite element method
(SGFEM)~\cite{ghanemspanos91, MR2084236, MR3202242} has emerged as
an efficient and rapidly convergent alternative to traditional Monte Carlo sampling.
However, the implementation of the SGFEM requires the solution of huge (although highly-structured) linear systems.
For realistic applications, such linear systems can only be solved using iterative methods
equipped with effective, bespoke preconditioners.}
\abrev{To that end, a range of linear algebra techniques have been employed including the}
multigrid and multilevel methods~\cite{ghanem03}, \cite{elmanfurnival07}, \cite{vandewalle07},
\cite{brezina14},~\cite{lee16},~\cite{pultarova17},~\cite{elmansu18},
domain decomposition methods~\cite{ghanem09}, \cite{loisel14}, \cite{waad14},
hierarchical methods~\cite{PellissettiG_00_ISS}, \cite{ghanem14}, \cite{ghanem14a},
as well as Krylov methods~\cite{ghanemkruger96}, \cite{PellissettiG_00_ISS},
\cite{elmanpowell2009}, \cite{jin09}, \cite{vandewalle10}, \cite{Ullmann10}.

In this work, we focus on Krylov methods; in particular, for
\abrev{a parametric elliptic PDE problem} %with random data
with solution approximated by the SGFEM, we
\abrev{employ} iterative methods of Krylov subspace type for which
we design and analyze a suitable class of preconditioners. 

As a model problem we consider the parametric steady-state diffusion equation
\abrev{subject to homogeneous Dirichlet boundary conditions}: 
\begin{equation}
\begin{aligned}-\nabla\cdot\left(a(\x,\y)\nabla u(\x,\y)\right) & =f(\x), &  & \x\in \Omega,\ \y\in\bGamma,\\
u(\x,\y) & =0, &  & \x\in\partial \Omega,\ \y\in\bGamma,
\end{aligned}
\label{eq:strong:form}
\end{equation}
where
$\Omega\subset\mathbb{R}^{d}$ ($d=1,2,3$) is a bounded (spatial)
domain with Lipschitz polygonal boundary $\partial\Omega$, 
$f\in H^{-1}(\Omega)$, and
$\bGamma := \prod_{m=1}^{\abrev{\infty}}\Gamma_m$ is the parameter domain with bounded
intervals $\Gamma_m\subset\RR$, $m\in\NN$.
%
%where $\Omega\subset\RR^d$, $d=1,2,3$, is a bounded (spatial)
%domain with Lipschitz polygonal boundary $\partial \Omega$ and $\bGamma:=\prod_{m=1}^M\Gamma_m$
%is the parameter domain with bounded domains
%$\Gamma_m\subset\RR$.
%
We also note that $\nabla$ denotes the spatial gradient operator $\nabla_{\x}$.
%The random diffusion coefficient $a$ will be assumed to be
%affine-parametric, i.e., we assume that
%\begin{equation}
%  \label{eq:affine0}
%  a(\x,\y)=a_0(\x)+\sum_{m=1}^Ma_m(\x)y_m,~~~(\x,\y)\in \Omega\times\bGamma.
%\end{equation}

The SGFEM applied to problem~\Refx{eq:strong:form} generates approximations in tensor product spaces $X \otimes S$,
where $X$ is a finite element space associated with the physical domain~$\Omega$, and
$S$ is a space of multivariate polynomials over a finite-dimensional manifold
\abrev{$\bGamma_M \subset \bGamma$
% in the parameter domain~$\bGamma$
(here, $M \in \NN$ refers to the number active parameters in the SGFEM approximation)}.
A typical SGFEM discretization of problem \Refx{eq:strong:form}
yields a structured linear system $A\u=\f$ \abrev{with the coefficient matrix}
%, where
\[
  A=
  \begin{bmatrix}
    A_{11} & A_{12} &\cdots & A_{1\Ny}\\
    A_{21} & A_{22} &\cdots & A_{2\Ny}\\
    \vdots& \vdots & \ddots & \vdots\\
    A_{\Ny1}& A_{\Ny2}& \cdots & A_{\Ny \Ny}
  \end{bmatrix},
\]
\abrev{where the}
blocks $A_{ij}$ are $\Nx \times \Nx$ matrices
\abrev{with $\Nx = \dim(X)$ and $\Ny = \dim(S)$}.
%the dimensions of the spatial and parametric finite element subspaces
%used in the discretisation, respectively.
\abrev{If the parametric coefficient $a(\x,\y)$ in~\Refx{eq:strong:form} is represented via
a (truncated or infinite) series expansion that is \emph{affine} in parameters, e.g.,
\[
%\begin{equation}
%  \label{eq:affine0}
   a(\x,\y) = a_0(\x)+\sum_{m=1}^{\infty} a_m(\x) y_m,\quad
   \x \in \Omega,\ \y \in \bGamma,
%\end{equation}
\]
%
%(e.g., using the Karhunen--Lo{\`e}ve or generalized polynomial chaos expansion of the underlying random field),
then it is well known (see, e.g., \cite[Chapter~9]{lord14}) %,~\cite{Ullmann10})
that the system matrix $A$ can be written as}
%
%If the parametric spaces $\Gamma_m$ are
%identical and if a spectral discretisation approach of degree $k$ is used for all
%parametric variables $y_m$ (cf. \cite{ghanemspanos91}), the
%system matrix can be seen to be
%
a sum of Kronecker products:
\begin{equation}
   A=G_0\otimes K_0+\sum_{m=1}^{\abrev{M}} G_m\otimes K_m.
   \label{eq:kron:sum}
\end{equation}
\abrev{Here,
%$N \in \NN$ depends on the structure of the polynomial space $S$ and on the type of the expansion of $a(\x,\y)$,
%\abrevx{$M$ is the number of active parameters in the SGFEM approximation,}
$G_m \in \RR^{\Ny \times \Ny}$ are the (parametric) matrices built from polynomial basis functions in $S$,}
%where
%$$\Ny={M+k \choose k}$$
and $K_m \in \RR^{\Nx \times \Nx}$ are the (spatial) stiffness matrices %of size $\Nx\times \Nx$
\abrev{associated with coefficients $a_m(\x)$ in the series expansion of~$a(\x,\y)$}.
%one of these stiffness matrices, denoted above by $K_0$,
%is associated with the mean component $a_0(\x) := \mathbb{E}[a](\x)$}. %of~$a(\x,\y)$}.
%
%\rev{We note that in the case of \todo{some} non-affine coefficient expansions,
%the structure of the system matrix is similar to the one in~\Refx{eq:kron:sum}
%(see~\S\ref{sec:numerics:non-affine} for the case of lognormal random field).}
%\[
%[K_m]_{k\ell}=\int_\Omega a_m(\x)\nabla\phi_\ell\cdot\nabla\phi_k\dd\x.
%\]

\abrev{The numerical solution of stochastic Galerkin linear systems
presents significant challenges.
On the one hand,}
it is evident from the structure of $A$ indicated above that such matrices can reach
huge sizes very quickly.
\abrev{For example, if $S$ is the space of complete polynomials of degree $\le k$ in $M$ parameters,
then $\Ny={M+k \choose k}$}
% particularly when we note that $\Ny$
grows very fast with $M$ and $k$.
% the number $M$ of parametric variables. While
On the other hand, in the case of affine-parametric expansion of the coefficient~$a(\x,\y)$ as given above,
% (e.g., for parametric representations stemming from
%the Karhunen--Lo{\`e}ve or other \emph{affine} expansions of random fields),
the matrix $A$ is block-sparse due to the sparsity of matrices $G_m$ in~\Refx{eq:kron:sum}.
This feature, however, is not guaranteed for other parametric representations of~$a(\x,\y)$
(see~\S\ref{sec:numerics:non-affine} for one example of such a representation).
Thus, the availability of effective preconditioning techniques is of paramount importance, in order to
enable the application of the SGFEM to a range of parametric PDE problems. %Stochastic Galerkin Finite Element Method.

In an early effort to provide an efficient solver technique for stochastic Galerkin linear systems, %is the
the \emph{mean-based preconditioner} \abrev{was proposed by Ghanem and Kruger in~\cite{ghanemkruger96}
and subsequently analyzed by Powell and Elman in~\cite{elmanpowell2009}}.
In the notation employed above, this preconditioner is defined as
\begin{equation}
  \label{eq:mb}
  P_0 := G_{0}\otimes K_{0}.
\end{equation}
It has been shown that 
under certain standard boundedness conditions on the diffusion coefficient $a(\x,\y)$,
the performance of \abrev{the conjugate gradient (CG) method} equipped
with the preconditioner $P_0$ is independent of the \abrev{dimensions $\Nx$ and $\Ny$ of the
underlying spatial and parametric approximation spaces}.
%of the spatial size $\Nx$
%of the problem and also independent of the number $M$ of parameters.
%
This is essentially due to the mean component $a_0(\x)$
strongly dominating other terms in the expansion of~$a(\x,\y)$. %\Refx{eq:affine0}.
When this is not the case, the performance deteriorates and dependence on \abrev{$\Ny$} may arise
\abrev{(e.g., via dependence on the number $M$ of active parameters and/or the polynomial degree $k$
of parametric approximations)}.
%Another parameter that affects
%performance is the spectral degree $k$ of approximation for the parametric
%variables $y_m$.

An alternative \abrev{approach} that takes into account contributions from
\abrev{all component matrices in~\Refx{eq:kron:sum}}
%terms $a_m$ in \Refx{eq:affine0}
was suggested by Ullmann in~\cite{Ullmann10}.
This preconditioner, which we denote by $P_\otimes$, is also defined as a Kronecker product:
\begin{equation}
  \label{eq:kron}
  P_{\otimes} := G\otimes K_{0},
\end{equation}
where the matrix \abrev{$G \in \RR^{\Ny \times \Ny}$} is constructed in order to minimize the Frobenius
norm of the difference between the system matrix and the preconditioner, i.e.,
$$G:=\arg\min_{Q}\norm{A - Q\otimes K_{0}}{F}.$$
While the eigenvalue bounds for the preconditioned system derived in~\cite{Ullmann10}
are not sharp
\abrev{and one cannot generally expect the iteration counts of the $P_{\otimes}$-preconditioned CG
to be independent of the dimension $\Ny$ of the parametric approximation space~$S$, the
\emph{Kronecker product preconditioner}} $P_\otimes$ outperforms the mean-based preconditioner $P_0$,
\abrev{particularly in the case of the approximation space $S$ comprising polynomials
of large degree $k$.}
%, while exhibiting optimality with respect to the spatial
%size $\Nx$ of the problem and \abrev{the} number $M$ of \abrev{active parameters in parametric approximation}
%%random variables.
%However, the spectral degree $k$ of
%approximation of parametric variables remains, in this case also, a parameter that affects performance.

\rev{
A preconditioning strategy that exploits the hierarchical structure of stochastic Galerkin matrices
was proposed by Soused{\' i}k and Ghanem in~\cite{ghanem14}.
In this strategy, the inverses of submatrices % of $A$
are \emph{approximated} by inverses of their diagonal blocks
in the action of a hierarchical symmetric block Gauss--Seidel preconditioner.
This preconditioner is further enhanced in~\cite{ghanem14} by performing the matrix-vector multiplications in its action
using only a subset of component matrices in~\Refx{eq:kron:sum} selected according to the size of the norm of stiffness matrices $K_m$.
In particular, a monotonic decay of the norms of $K_m$ effectively results in \emph{truncating} the sum in~\Refx{eq:kron:sum}.
%The results of extensive numerical experiments reported in~\cite{ghanem14}
Extensive numerical experiments
for a model problem with truncated lognormal diffusion coefficient have demonstrated the effectiveness and competitiveness
of this combined preconditioning strategy (called the \emph{truncated hierarchical preconditioning})
in terms of both convergence of iterations and computational cost.
The results of these experiments have also shown that
truncated versions of the non-hierarchical symmetric block Gauss--Seidel preconditioner and
the approximate hierarchical Gauss--Seidel preconditioner
are largely comparable in terms of convergence of~iterations.}

In this paper, we \rev{study} a preconditioning technique based on
truncating the \abrev{sum of Kronecker products in~\Refx{eq:kron:sum} as follows:}
%expression \Refx{eq:affine0} of the diffusion coefficient.
%\begin{equation} \label{eq:trunprec}
\[
  P_r := G_0\otimes K_0 + \sum_{m=1}^r G_m\otimes K_m.
\]
%\end{equation}
We will refer to this class of solvers as \emph{truncation preconditioners.}
%This can be seen to 
\abrev{While it includes the mean-based preconditioner as a special case,
by capturing additional significant components of the stochastic Galerkin matrix $A$ \rev{one aims} % we aim
to improve the preconditioner's performance
retaining its optimality with respect to the discretization parameters (i.e., $\Nx$, $M$ and $k$).}
%The aim is to incorporate additional
%information in the preconditioner in order to further improve
%performance, possibly eliminate dependence on all the discretisation
%parameters ($\Nx$, $M$ and $k$).
%
\rev{
Truncation preconditioners of this type were considered in~\cite[Section~4.2]{KubinovaP_20_BPS}
and analyzed therein for two extreme cases, namely $r = 0$ and~$r = M-1$.
While the preconditioning matrix $P_r$ has a block-diagonal structure in the case of the \emph{tensor-product} polynomial space~$S$
(with appropriately ordered basis functions; see~\cite{KubinovaP_20_BPS}),
this property, in general, does not hold for the matrix $P_r$ if $S$ is the space of \emph{complete} polynomials.
Therefore, in the latter case, considered in the present paper, a practical implementation of the truncation preconditioner $P_r$
requires an additional technique to ensure efficient application of the action of $P_r^{-1}$ on a given vector.
To this end, similarly to the strategy in~\cite{ghanem14},
we propose to use a symmetric block Gauss--Seidel approximation $\tilde P_r$ of the truncation preconditioner $P_r$.
}

\rev{
Focusing on the case of affine-parametric representation of~$a(\x,\y)$, % the diffusion coefficient,
our main goal in this paper is to perform spectral analysis of the preconditioned matrices and
%also include a generic analysis that allows to 
establish optimality of the preconditioners $P_r$ and $\tilde P_r$
%our preconditioning technique
with respect to all discretization parameters.
By doing this we fill a gap in the theoretical analysis of preconditioning techniques
for the numerical solution of stochastic Galerkin linear systems.
}

The paper is structured as follows. In the next section we present a
detailed problem formulation, including
\abrev{specific assumptions on the parametric diffusion coefficient $a(\x,\y)$},
the variational formulation of~\Refx{eq:strong:form}, as well as
the definitions and properties of the matrices $A_{ij}$, $G_m$, $K_m$.
In section~\ref{sec:preconditioners}, we introduce \rev{and analyze} a class of
preconditioners $P_r$ based on truncation of the series representation
of the parametric diffusion coefficient~$a(\x,\y)$.
A symmetric block Gauss--Seidel
approximation to $P_r$ is introduced and analyzed in section~\ref{sec:modified:precond}.
Section~\ref{sec:numerics} includes a range of numerical results, with
conclusions and potential extensions summarized in section~\ref{sec:conclusions}.

%\newpage
%----------section 2----------------------------------------------------
\section{Problem formulation}  \label{sec:problem}

In this section we outline some standard
background results and assumptions and introduce the variational formulation
required for the numerical solution of \Refx{eq:strong:form} via the SGFEM.

%Consider \todo{again} the homogeneous Dirichlet problem for the steady-state
%diffusion equation with parametric coefficient:
%\begin{equation}
%\begin{aligned}-\nabla\cdot\left(a(\x,\y)\nabla u(\x,\y)\right) & =f(\x), &  & \x\in \Omega,\ \y\in\bGamma,\\
%u(\x,\y) & =0, &  & \x\in\partial\Omega,\ \y\in\bGamma,
%\end{aligned}
%\label{eq:strong:form}
%\end{equation}
%where
%$\Omega\subset\mathbb{R}^{d}$ ($d=1,2,3$) is a bounded (spatial)
%domain with a Lipschitz polygonal boundary $\partial\Omega$, 
%$f\in H^{-1}(\Omega)$, and
%$\todo{\bGamma := \prod_{m=1}^\infty\Gamma_m}$ is the parameter domain with bounded
%$\Gamma_m\subset\RR$, $m\in\NN$.

%--------section2.1----------------------------------------------------
\subsection{Functional analytic framework} \label{sec:func:frame}

%Let $\Gamma_m\subset\RR$ with $m\in\NN$ and define
%\begin{equation}
%  \label{eq:bgamma}
%  \bGamma:=\prod_{m=1}^\infty\Gamma_m.
%\end{equation}

\abrev{Let $y_m\in\Gamma_m$ be} the images of independent bounded random variables with cumulative
density function $\pi_{m}(y_{m})$ and probability density function
$p_{m}(y_{m})=\dd\pi_{m}(y_{m})/\dd y_{m}$. The
joint cumulative density function and the joint probability density
function of the associated multivariable random variable $\y\in\bGamma$ are defined,
respectively, as
\[\pi(\bm{y}):=\prod_{m=1}^{\infty}\pi_{m}(y_{m})\quad \text{and}\quad
  p(\bm{y}):=\prod_{m=1}^{\infty}p_{m}(y_{m}).
\]
Without loss of generality, we assume for all $m\in\NN$ that $\Gamma_{m}:=[-1,1]$\abrev{; additionally, we assume} that $p_{m}$ is even. Note that each $\pi_{m}$ is a probability measure
on $(\Gamma_{m},\mathcal{B}(\Gamma_{m}))$, where $\mathcal{B}(\Gamma_{m})$
is the Borel $\sigma$-algebra on $\Gamma_{m}$. Accordingly, $\pi$
is a probability measure on $(\bGamma,\mathcal{B}(\bGamma))$, where
$\mathcal{B}(\bGamma)$ is the Borel $\sigma$-algebra on $\bGamma$.
Then $L_{\pi_{m}}^{\abrev{2}}(\Gamma_{m})$, $L_{\pi}^{\abrev{2}}(\bGamma)$
% with $1\le r\le\infty$
represent the weighted Lebesgue spaces with associated inner products
\begin{align*}
  \langle f,g\rangle_{\pi_{m}} & := \int_{\Gamma_{m}}p_{m}(y_{m})f(y_{m})g(y_{m})\dd y_{m}, \qquad
  f,g\in L_{\pi_{m}}^{2}(\Gamma_{m}),
  \\[3pt]
  \langle f,g\rangle_{\pi} & := \int_{\bGamma}p(\bm{y})f(\bm{y})g(\bm{y})\dd\bm{y}, \qquad
  f,g\in L_{\pi}^{2}(\bGamma).
\end{align*}
Finally, a space relevant to the
weak formulation of problem (\ref{eq:strong:form}) is
$L_{\pi}^{2}(\bGamma;H_{0}^{1}(\Omega)),$
which is the space of strongly measurable functions
$v:\Omega\times\bGamma\to\mathbb{R}$ such that 
\[
   \norm{v}{L_{\pi}^{2}(\bGamma;H^1_0(\Omega))} :=
   \norm{\norm{v(\cdot,\bm{y})}{H^1_0(\Omega)}}{L_{\pi}^{2}(\bGamma)} :=
   \bigg[
   \int_\bGamma
   p(\y) \norm{v(\cdot,\bm{y})}{H^1_0(\Omega)}^2\dd\y
   \bigg]^{1/2}
   < +\infty,
\]
where $H_{0}^{1}(\Omega)$ is the usual Sobolev
space of functions in $H^{1}(\Omega)$ that vanish at the boundary $\partial\Omega$
in the sense of traces. It is known that $L_{\pi}^{2}(\bGamma;H_{0}^{1}(\Omega))$
is isometrically isomorphic to the following tensor product Hilbert space (see \cite[Remark C.24]{schwabgittelson2011}):
\begin{equation}
V:=L_{\pi}^{2}(\bGamma)\otimes H_{0}^{1}(\Omega).\label{eq:tensorspace}
\end{equation}
We will use the space $V$ to \abrev{derive} the variational formulation of problem~(\ref{eq:strong:form})
in~\S\ref{sec:weak:and:Galerkin} below.
Before doing this, let us make some specific assumptions on the parametric diffusion coefficient~$a(\x,\y)$.

%------section2.2---------------------------------------------------

\subsection{The parametric diffusion coefficient} \label{sec:diff:coeff}

We will assume that the diffusion coefficient $a$ is \emph{affine-parametric}, i.e.,
\begin{equation}
a(\x,\y)=a_{0}(\x)+\sum_{m=1}^{\infty}a_m(\x)y_{m},\quad\x\in \Omega,\ \y\in\bGamma,\label{eq:affine}
\end{equation}
with $a_{m}\in L^{\infty}(\Omega)$, $m \in \NN_0$, and with the series
converging uniformly in $L^{\infty}(\Omega)$.
The representation~\Refx{eq:affine} is motivated by the Karhunen--Lo{\`e}ve
expansion of a second-order random field $a$ with given mean $\mathbb{E}[a]$
and covariance function $\hbox{\rm Cov}[a](\x,\x')$. % = \mathbb{E} \big[ (a(\x)-a_{0}(\x)) (a(\x')-a_{0}(\x')) \big]$.
In this case, $a(\x,\y)$ is represented as in~(\ref{eq:affine})
with $a_{0}(\x)=\mathbb{E}[a]$ and $a_{m}(\bm{x})=\sqrt{\lambda_{m}}\varphi_{m}(\bm{x})$
($m=1,2\ldots$), where $\{(\lambda_{m},\varphi_{m})\}_{m=1}^{\infty}$
are the eigenpairs of the integral operator $\int_{\Omega}\hbox{\rm Cov}[a](\x,\x')\varphi(\x')\dd\x'$
such that $\lambda_{1}\geq\lambda_{2}\geq\ldots>0$, and $y_{m}$
($m=1,2\ldots$) are the images of pairwise uncorrelated random variables
with zero mean and unit variance (see, e.g.,~\cite[\S2.3]{ghanemspanos91}).

As is the case for deterministic diffusion problems, the standard
conditions for well-posedness of the weak formulation of problem~\Refx{eq:strong:form} are the positivity and boundedness of the diffusion
coefficient~$a$.
%Thus, we will assume that
%\begin{equation}
%0<a_{\min}\leq a(\x,\y)\leq a_{\max}<\infty\quad\text{a.e. in }\Omega\times\bGamma.\label{eq:assumption:on:a}
%\end{equation}
%for some positive constants $a_{\min}$ and $a_{\max}$.
%
In order to ensure that the coefficient~$a(\x,\y)$ given by~(\ref{eq:affine})
satisfies these conditions, %~(\ref{eq:assumption:on:a}),
we assume that (cf.~\cite[Proposition~2.22]{schwabgittelson2011})
\begin{equation}
a_{0}^{\min}\leq a_{0}(\x)\leq a_{0}^{\max}\quad\text{a.e. in }\Omega\label{eq:assumption:on:a0}
\end{equation}
for some constants $0 < a_0^{\min} \le a_0^{\max} < \infty$
and that  % we need only $a_{0}\in L^{\infty}\left(\Omega\right)$ as in (\ref{eq:assumption:on:a0}) and
%$a_{i}\in L^{\infty}(\Omega)$ for $i\in\mathbb{N}$, satisfying 
\begin{equation}
   \tau:=\frac{1}{a_{0}^{\min}} \bigg\| \sum_{m=1}^{\infty}|a_{m}| \bigg\|_{\infty}<1,
   \label{eq:assumption:on:ai-1}
 \end{equation}
where $\|\cdot\|_{\infty}$ denotes the norm in $L^{\infty}(\Omega)$.
In this case, an elementary calculation shows~that
\begin{equation}
0<a_{\min}\leq a(\x,\y)\leq a_{\max}<\infty\quad\text{a.e. in }\Omega\times\bGamma
\label{eq:assumption:on:a}
\end{equation}
with
%\begin{equation} \label{eq:amin:amax}
$a_{\min} := a_0^{\min} (1 - \tau)$ and %,\qquad
$a_{\max} := a_0^{\max} + a_0^{\min} \tau$.
%\end{equation}

\begin{remark}
\rev{As an alternative to \Refx{eq:assumption:on:ai-1}, one can make a weaker assumption:
\begin{equation} \label{eq:assumption:on:ai-1:KubPul}
   \widetilde\tau := \bigg\| a_{0}^{-1} \sum_{m=1}^{\infty} |a_{m}| \bigg\|_{\infty} < 1.
\end{equation}
In this case, one has
%\begin{equation} \label{eq:assumption:on:a:KubPul}
%   0< a_0^{\min} (1-\tilde\tau) \le a_{0}(x) (1-\tilde\tau) \leq a(\x,\y) \leq a_{0}(x) (1+\tilde\tau) \leq a_{0}^{\max} (1+\tilde\tau) < \infty
%\end{equation}
%almost everywhere in $\Omega\times\bGamma$.
\begin{equation} \label{eq:assumption:on:a:KubPul}
   0< \tilde a_{\min} \le a_{0}(\x) (1-\tilde\tau) \leq a(\x,\y) \leq a_{0}(\x) (1+\tilde\tau) \leq \tilde a_{\max} < \infty\ \
   \text{a.e. in }\Omega\times\bGamma
\end{equation}
with $\tilde a_{\min} := a_0^{\min} (1-\tilde\tau)$ and $\tilde a_{\max} := a_{0}^{\max} (1+\tilde\tau)$.
We refer to~{\rm \cite[Section~2.3]{KubinovaP_20_BPS}} for a detailed discussion of assumption~\Refx{eq:assumption:on:ai-1:KubPul}
and its comparison to the one in~\Refx{eq:assumption:on:ai-1}.}
\end{remark}

%\todo{We note that \Refx{eq:assumption:on:a0} and \Refx{eq:assumption:on:ai-1}
% are sufficient for \Refx{eq:assumption:on:a} to hold, but \abrev{are} not
% necessary. Indeed, we will consider examples in Section~\ref{sec:numerics} where~\Refx{eq:assumption:on:ai-%1} does not hold,
% while problem~\Refx{eq:strong:form} is still well-posed (TODO).
%}
%-----------section 2.3-------------------------------------------------------------
\subsection{Variational formulation and Galerkin approximations} \label{sec:weak:and:Galerkin}

Let $V$ be defined as in \Refx{eq:tensorspace}. 
The weak formulation of problem~(\ref{eq:strong:form}) reads as
follows: 
find $u \,{\in}\, V$ such~that
\begin{equation}
\A(u,v)=\F(v)\quad \forall\, v\in V,\label{eq:weak:form}
\end{equation}
where the bilinear form $\A:V\times V\to\RR$ and the linear functional
$\F:V\to\RR$ are defined by 
\begin{align}
\A(v,w) & :=\int_{\bGamma}p(\y)\int_{\Omega}a(\x,\y)\nabla v(\x,\y)\cdot\nabla w(\x,\y)\dd\x \dd\y,
\label{eq:A-form}\\[4pt]
\F(w) & :=\int_{\bGamma}p(\y)\int_{\Omega}f(\x)w(\x,\y)\dd\x \dd\y.
\nonumber
\end{align}
The boundedness \rev{of the parametric coefficient $a$ (see~(\ref{eq:assumption:on:a}) and~\Refx{eq:assumption:on:a:KubPul})}
implies a unique solvability of~(\ref{eq:weak:form}) due to the Lax--Milgram theorem.

We discretize problem~(\ref{eq:weak:form}) by using the Galerkin
projection onto a finite-dimen\-sio\-nal subspace of~$V$
\abrev{constructed as a tensor product of finite-dimensional subspaces} of $H^1_0(\Omega)$ and $L^2_\pi(\bGamma)$.
Let us describe these subspaces in detail.
Let ${\cal T}_h$ denote a conforming, quasi-uniform partition of $\Omega$ into simplices $K$ of
maximum diameter $h$. We define
\[
   X_h \,\abrev{:=}\, \seq{\phi:\Omega\rightarrow\RR : \phi{\mid_K}\in \PP_{q}(K)} \cap C^0(\Omega)
   \;\abrev{\subset}\; H^1_0(\Omega),
\]
where $\PP_{q}(K)$ denotes the space of polynomials of degree $q$ defined on $K$.
%In other words, we choose $X_h$ to be the
%standard space of continuous piecewise polynomial functions defined on
%the partition ${\cal T}_h$ of $\Omega$.
We will assume that
\[
X_h=\spn\seq{\phi_{1},\phi_{2},\ldots,\phi_{\Nx}},
\]
where $\phi_j(\x)$ have local support and $\Nx := \dim(X_h)$.

Let us now introduce a polynomial subspace of $L^2_\pi(\bGamma)$.
To this end, for each $m \in \NN$, we first consider the sequence $\{P_j^m(y_m) :\, j \in \NN_0\}$
of univariate polynomials that are orthonormal with respect to 
the inner product $\langle\cdot,\cdot\rangle_{\pi_{m}}$,
such that $P_{j}^{m}$ is a polynomial of degree $j\in\NN_{0}$.
These polynomials form an orthonormal basis of $L_{\pi_{m}}^{2}(\Gamma_{m})$, i.e.,
$L^2_{\pi_m}(\Gamma_m)=\spn\seq{P^m_j :\, j\in\NN_0}.$
Moreover, %it is known that
they satisfy the following three-term
 recurrence \cite[Theorem~1.29]{gautschi04}:
 \begin{equation}
   {c^m_{j+1}}P_{j+1}^m(y_m)=(y_m-a^m_j)P_j^m(y_m)-{c^m_j}P_{j-1}^m(y_m),\quad j=0,1,2,\ldots\label{eq:3term}
 \end{equation}
   with $P_{-1}^m(y_m)=0,~P_0^m(y_m)=1/{c^m_0}$, where
   $a^m_j,c^m_j$ are defined in terms of inner-products involving
   the monic orthogonal polynomial counterparts to $P^m_j$ (for
   details, see \cite[Chapter~1]{gautschi04}). For our choice of intervals
   $\Gamma_m=[-1,1]$, the recurrence coefficients $a^m_j,\, c^m_j$ are bounded;
   in particular $a^m_j = 0$ (since the measure $\pi_m$ is symmetric) and
   there holds (see \cite[Theorem~1.28]{gautschi04})
   \begin{equation}
     \label{eq:cmj}
     0<c^m_j\leq 1,\quad m\in\NN,\ \ j=0,1,2,\ldots.
   \end{equation}
Consider now the \emph{index set} of multi-indices $\balpha$ with finite support
%\begin{equation} \label{eq:index:set}
\[
   \I:=\seq{\balpha=(\alpha_{1},\alpha_{2},\ldots)\in\NN_{0}^{\NN} :\, \max(\spp\balpha)<\infty},
\]
%\end{equation}
where $\spp\balpha=\{m\in\NN;\;\alpha_{m}\neq0\}$. 
For each multi-index $\balpha\in\I$, we define the multivariate
polynomial 
\[
\psi_{\balpha}(\y):=\prod_{m\in\spp\balpha}P_{\alpha_{m}}^{m}(y_{m}).
\]
The set $\{\psi_{\balpha} : \balpha\in\I\}$ is an orthonormal basis
of $L_{\pi}^{2}(\bGamma)$,
see~\cite[Theorem~2.12]{schwabgittelson2011}.
Thus, any finite index set $\I_n\subset\I$ \abrev{induces} a finite
dimensional subspace
$\spn\seq{\psi_{\balpha} :\, \balpha\in\I_n}$ of $L_{\pi}^{2}(\bGamma)$.
In this paper, we employ finite index sets of the following type:
\[
  \I_{k}^{M}:=\seq{\balpha\in\I :\, |\balpha|\leq k\text{ and }\alpha_{m}=0\ \;\forall\,m>M},\quad
  k \in \NN_0,\ \ M \in \NN,
\]
where $|\balpha|=\sum_{m\in\spp\balpha}\alpha_{m}$.
We denote the corresponding finite-dimensional subspaces of $L_{\pi}^{2}(\bGamma)$~as
\begin{equation}
  \label{eq:skm}
  \SkM:=\spn\seq{\psi_{\balpha} : \balpha\in\I_{k}^{M}}.
\end{equation}
Thus, $\SkM$ is the space of complete polynomials of degree $\le\,k$ in $M$ variables;
its dimension is given by
\[
   \Ny:=\dim(\SkM)=\card{\I_{k}^{M}}={M+k \choose k}.
\]
Furthermore, 
there exists a bijection $\bkappa:\{1,2,\ldots,\Ny\}\to\I_{k}^{M}$, so that we can also
describe $\SkM$ via the span
$\SkM=\spn\seq{\psi_{\bkappa(j)} :\, 1\leq j\leq \Ny}.$

We can now define the following finite-dimensional subspace of $V$:
\begin{equation}
  \label{eq:VhkM}
  \VhkM \,{:=}\, X_{h} {\otimes} \SkM \,{=}\,
  \spn\!\big\{\varphi_{ij}(\x,\y) \,{:=}\, \phi_i(\x)\psi_{\bkappa(j)}(\y) :\,\!
  1 \,{\leq}\, i \,{\leq}\, \Nx,\ 1 \,{\leq}\, j \,{\leq}\, \Ny\big\}.
\end{equation}
The resulting discrete formulation of \Refx{eq:weak:form} reads: find $\uhkM\in\VhkM$
such that %for all $v\in V_{hk}^M$
\begin{equation}
\A(\uhkM,v)=\F(v)\quad \forall\, v \in V_{hk}^M.\label{eq:sGFEM}
\end{equation}
Using the definition of %the basis elements for
$\VhkM$ in \Refx{eq:VhkM} we write the Galerkin approximation $\uhkM$~as
%follows: 
\begin{equation}
\uhkM(\x,\y)=\sum_{i=1}^{\Nx}\sum_{j=1}^{\Ny}u_{ij}\varphi_{ij}(\x,\y).\label{eq:approx:sol}
\end{equation}
    The Galerkin projection onto the finite-dimensional space $\VhkM$
    defined via the choice of finite index set $\I_k^M$
    can be shown to correspond to a discrete weak formulation
    involving a truncation at $m=M$ of the parametric diffusion
    coefficient $a$ given in~\Refx{eq:affine}. More precisely, if we let
%\begin{equation} \label{eq:affineM}
\[
   a_M(\x,\y):=a_{0}(\x)+\sum_{m=1}^Ma_m(\x)y_{m},\quad\x\in \Omega,\ \y\in\bGamma,
\]
%\end{equation}
then the associated bilinear form
\[\A_M(v,w) :=\int_{\bGamma}p(\y)\int_{\Omega}a_M(\x,\y)\nabla
  v(\x,\y)\cdot\nabla w(\x,\y)\dd\x \dd\y
\]
satisfies (see, e.g.,~\cite[p.~A349]{BespalovPS_14_ENA})
\begin{equation}
\A_M(v,w)=\A(v,w)\quad\forall v,w\in \VhkM.\label{eq:AM}
\end{equation}
%
% We note that by using the Galerkin projection onto the finite-dimensional
% space $\VhkM$, the infinite sum in the representation~(\ref{eq:affine})
% of the coefficient $a(\x,\y)$ is naturally truncated at $m=M$
% (see, e.g.,~\cite[p.~A349]{BespalovPS_14_ENA}), and hence the Galerkin
% approximation $\uhkM$ is indeed computable. In particular,
\abrev{Using representation~(\ref{eq:approx:sol}) of the Galerkin solution,
and setting $v=\varphi_{st}$ in~(\ref{eq:sGFEM})
for $s=1,\ldots,N_{\x}$ and $t=1,\ldots,N_{\y}$, we obtain the following
  linear system:}
%\begin{equation}
\[
   \sum_{i=1}^{\Nx}\sum_{j=1}^{\Ny} u_{ij}\A(\varphi_{ij},\varphi_{st})=\F(\varphi_{st}).
%   \label{eq:sgfem1}
\]
%\end{equation}
\abrev{This system becomes, using \Refx{eq:AM} and the separable form \Refx{eq:VhkM} of
$\varphi_{ij}$, $\varphi_{st}$ and of
each term in the series expansion~(\ref{eq:affine}) of the diffusion coefficient $a(\x,\y)$,}
%\begin{equation}
\[
   \sum_{m=0}^M\sum_{i=1}^{N_{\x}} \sum_{j=1}^{N_{\y}}
   u_{ij} \int_{\Omega} a_{m}\nabla\phi_{i}\cdot\nabla\phi_{s} \dd\x
   \int_{\bGamma} y_{m} \psi_{\bkappa(j)} \psi_{\bkappa(t)}p \dd\y =
   \int_{\Omega}f\phi_{s}\dd\x\int_{\bGamma}\psi_{\bkappa(t)}p\dd\y,
%   \label{eq:sgfem2}
\]
%\end{equation}
where we set $y_0=1$.
Therefore, the discrete formulation~(\ref{eq:sGFEM}) yields a
linear system $A\u=\f$ with block structure. Specifically, the
coefficient matrix $A$, the solution vector $\u$, and the right-hand
side vector $\f$ are given by 
\begin{equation}
A=\left[\begin{array}{cccc}
A_{11} & A_{12} & \cdots & A_{1N_{\y}}\\
A_{21} & A_{22} & \cdots & A_{2N_{\y}}\\
\vdots & \vdots & \ddots & \vdots\\
A_{N_{\y}1} & A_{N_{\y}2} & \cdots & A_{N_{\y}N_{\y}}
\end{array}\right],\quad \u=\left[\begin{array}{c}
\u_{1}\\
\u_{2}\\
\vdots\\
\u_{N_{\y}}
\end{array}\right], \quad \f=\left[\begin{array}{c}
\f_{1}\\
\f_{2}\\
\vdots\\
\f_{N_{\y}}
\end{array}\right],
\label{eq:A:u:f}
\end{equation}
respectively, where 
\[
A_{tj}=\langle\psi_{\bkappa(j)},\psi_{\bkappa(t)}\rangle_{\pi}\,K_{0}+\sum_{m=1}^{M}\langle y_{m}\psi_{\bkappa(j)},\psi_{\bkappa(t)}\rangle_{\pi}\,K_{m},\quad t,j=1,\ldots,N_{\y}
\]
with finite element (stiffness) matrices $K_{m}$, $m=0,1,\ldots,M$,
defined by 
\begin{align*}
\left[K_{m}\right]_{si} & := \int_{\Omega}a_{m}\nabla\phi_{i}\cdot\nabla\phi_{s}\,\dd\x,\quad s,i=1,\ldots,N_{\x},
\\[3pt]
\u_{j} & := \left[u_{1j}\;u_{2j}\;\ldots\;u_{N_{\x}j}\right]^{T},\quad j=1,\ldots,N_{\y},
\end{align*}
and 
\[
\left[\f_{t}\right]_{s}:=\langle1,\psi_{\bkappa(t)}\rangle_{\pi}\,\int_{\Omega}f\phi_{s}\,\dd\x,\quad
  s=1,\ldots,N_{\x},\ t=1,\ldots,N_{\y}.
%\begin{array}{l}
%  s=1,\ldots,N_{\x},\\ t=1,\ldots,N_{\y}.
%\end{array}
\]
Using Kronecker products for matrices, it is convenient to write the
coefficient matrix $A$ in the following \abrev{compact} form 
\begin{equation}
A=G_{0}\otimes K_{0}+\sum_{m=1}^{M}G_{m}\otimes K_{m},\label{eq:decomp:A}
\end{equation}
where 
\begin{equation}
\left[G_{0}\right]_{tj}:=\langle\psi_{\bkappa(j)},\psi_{\bkappa(t)}\rangle_{\pi},\quad\left[G_{m}\right]_{tj}:=\langle y_{m}\psi_{\bkappa(j)},\psi_{\bkappa(t)}\rangle_{\pi}\label{eq:G:matrices}
\end{equation}
for $m=1,\ldots,M$ and $t,j=1,\ldots,N_{\y}$.

The stochastic Galerkin matrix $A$ is symmetric and positive definite.
Furthermore, as it follows from the theorem below, $A$ is block sparse
with no more than $2M+1$ nonzero block matrices per row.

\begin{theorem} \label{thm:G:matrices} {\rm \cite[Theorems 9.58, 9.59]{lord14}} 
Consider the matrices $G_{m}$ defined
 in~\Refx{eq:G:matrices}
for $m=0,1,\ldots,M$. The matrix $G_{0}$ is the $N_{\y}\times N_{\y}$
identity matrix and each matrix $G_{m}$ for $m=1,2,\ldots,M$ has
\rev{at most} two nonzero entries per row. More precisely, 
\[
\left[G_{m}\right]_{tj}=%\left(\xi_{i}\psi_{\bkappa(s)},\psi_{\bkappa(t)}\right)_{p}=
\begin{cases}
c_{\gamma_{m}+1}^{m}, & \text{if \ }\gamma_{m}=\beta_{m}-1\text{ \ and \ }\gamma_{\ell}=\beta_{\ell}\ \ \forall\,\ell\in\NN\setminus\{m\},\\[2pt]
c_{\gamma_{m}}^{m}, & \text{if \ }\gamma_{m}=\beta_{m}+1\text{ \ and \ }\gamma_{\ell}=\beta_{\ell}\ \ \forall\,\ell\in\NN\setminus\{m\},\\[2pt]
0, & \text{otherwise,}
\end{cases}
\]
where $\bm{\gamma}=\bkappa(t)$, $\bm{\beta}=\bkappa(j)$ and
$c_{\gamma_{m}}^{m}$ \abrev{are the coefficients arising in the three-term
  recurrence \Refx{eq:3term} which defines
the orthonormal polynomials $P_j^m$.}
\end{theorem}

%---------------section 3-------------------------------------------------------
\section{Truncation preconditioners} \label{sec:preconditioners}
% -----------------------------------------------------

\rev{In this section, we consider 
% a generalization of the mean-based preconditioner defined in~\Refx{eq:mb}. In particular, we define 
a class of preconditioners that are induced by bilinear
forms associated with truncations of the series representation of the parametric diffusion coefficient $a$
(cf.~\cite{ghanem14, KubinovaP_20_BPS}).}
We will show that these \emph{truncated} bilinear forms are equivalent to the bilinear form
arising in the variational formulation of our PDE.
The immediate consequence of this fact is that \rev{the resulting \emph{truncation preconditioners}} will be
\abrev{optimal in some sense to be described below (see Definition~\ref{def:optimal})}.

%-----------------section 3.1------------------------------------------------------
\subsection{Equivalent bilinear forms and preconditioning} \label{sec:equiv:bilinear:forms}

A generic approach to preconditioner design for
discretizations of variational problems is based on
approximating the bilinear forms arising in the formulation of
the problem. \abrev{For symmetric and coercive problems, the
well-known concept
of equivalence of bilinear forms translates into spectral equivalence
between the coefficient matrix and the preconditioner induced by the
approximating bilinear form; in turn, spectral equivalence enables both the design and
analysis of effective preconditioning techniques.}
We summarize this approach in Proposition~\ref{prop:equiv} below,
which requires the following two definitions.

\begin{definition}
  We say that \abrev{positive definite} symmetric bilinear forms
  $\A, \B :V\times V\rightarrow\RR$ are equivalent if there exist positive constants
  $\uptheta,\Uptheta$ such that for all $v\in V$ there holds
  $$\uptheta\B(v,v)\leq \A(v,v)\leq \Uptheta\B(v,v).$$
\end{definition}

%\begin{remark}
%    Since positive definite bilinear forms on $V\,{\times}\,V$ induce norms
%    on $V$, the equivalence~of bilinear forms described in the above
%    definition amounts essentially to equivalence of norms on~$V$.
%\end{remark}

\begin{definition}
  We say that symmetric positive definite matrices $A,B\in\RR^{n\times n}$ are spectrally equivalent
  if
  there exist positive constants $\uptheta,\Uptheta$ independent of
  $n$ such that for all $\bv\in\RR^n$ there holds
  $$\uptheta \bv^TB\bv\leq \bv^TA\bv \leq \Uptheta\bv^TB\bv.$$
  \abrev{In this case, we write $A\sim B$}.
\end{definition}

\begin{remark}%\label{rem:equiv}
  The relation $\sim$ is an equivalence relation. In particular,
  transitivity will be relevant in our subsequent discussion.
\end{remark}

%If $A\sim B$ with constants of equivalence $\uptheta,\Uptheta$, then the
%following spectral inclusion holds
%\[
%  \Lambda(B^{-1}A)\subset[\uptheta,\Uptheta].
%\]
\abrev{Bilinear form equivalence is connected to the well-known concepts of operator and
spectral equivalence (see~\cite{dyakonov66},~\cite{fmp90}) as well as
norm-equivalent preconditioners (see~\cite{loghinwathen04}). In this
context, the following result is key to our subsequent analysis.}

\begin{proposition} \label{prop:equiv}
  \abrev{Let $\A,\B$ denote positive definite symmetric bilinear forms on
  $V\times V$ which are equivalent.}
  Let $V_n=\spn\seq{\varphi_1,\ldots,\varphi_n}\subset V$
  and let $A,B\in\RR^{n\times n}$ be defined as follows
  $$A_{ij}=\A(\varphi_j,\varphi_i),\qquad B_{ij}=\B(\varphi_j,\varphi_i).$$
  Then $A\sim B$ and the spectrum of
  $B^{-1}A$ satisfies
  \begin{equation}
  \Lambda(B^{-1}A)\subset[\uptheta,\Uptheta],\label{eq:eigbound}
\end{equation}
  where $\uptheta,\Uptheta$ are the constants of equivalence for
  $\A,\B$. 
\end{proposition}

The above result motivates the following definition.

\begin{definition}\label{def:optimal}
Let $A,B\in\RR^{n\times n}$ satisfy \Refx{eq:eigbound} with
constants $\theta,\Theta$ independent of $n$. Then $B$ is said to be
an optimal preconditioner for $A$ with respect to the problem size $n$.
\end{definition}

Preconditioner optimality translates into performance optimality. In
particular, it is well-known that the preconditioned Conjugate
Gradient algorithm applied to the linear system $A\u=\f$ with optimal
preconditioner $B$ converges in a number of steps independent of $n$.
Our aim is to construct optimal preconditioners with respect to
  the problem size for the coefficient matrix in \Refx{eq:A:u:f}.  We do this by first adapting the result of Proposition \ref{prop:equiv} to the parametric elliptic problem~\Refx{eq:strong:form}.
We will need the following auxiliary result.

\begin{lemma}\label{lem: equiv tool}  Let
  $p:\boldsymbol{\Gamma}\to\mathbb{R}_{+}$ and
  assume that $b_i:\Omega\times\bGamma \,{\to \RR_{+}}$ $(i=1,2)$ satisfy
  $$0<\beta_i^{\min}\leq b_i(\x,\y)\leq \beta_i^{\max}\quad
  {\text{a.e. in $\Omega\times\bGamma \ni (\x,\y)$}}.$$
  Define the bilinear
  forms $\B_i:V\times V \to \RR$ via
  $$\B_i(v,w)=\int_{\bGamma}p(\y)\int_{\Omega}b_{i}(\x,\y)\nabla
  v(\x,\y)\cdot\nabla w(\x,\y)\dd\x \dd\y,\quad {i = 1,2}.$$
  Then the bilinear forms $\B_i$ are equivalent:
  $$\uptheta\B_2(v,v)\leq \B_1(v,v)\leq
  \Uptheta\B_2(v,v)\quad \forall\, v \in V,$$
  where
  $$\uptheta=\frac{\beta_1^{\min}}{\beta_2^{\max}},\qquad \Uptheta=\frac{\beta_1^{\max}}{\beta_2^{\min}}.$$
\end{lemma}

\begin{proof}
  For any $v\in V$, we have 
  \begin{align*}
             \B_1(v,v) 
              & =\int_{\boldsymbol{\Gamma}}p(\y)\int_{\Omega}
              \left(\frac{b_{1}(\x,\y)}{b_{2}(\x,\y)}\right)b_{2}(\x,\y)\nabla v(\x,\y)\cdot\nabla v(\x,\y)\dd\x \dd\y
              \\[4pt]
              & \leq\left(
              {\operatorname*{ess\;sup}_{(\x,\y)\in \Omega\times\boldsymbol{\Gamma}}}\;
              \frac{b_{1}(\x,\y)}{b_{2}(\x,\y)}\right)
              \int_{\boldsymbol{\Gamma}}p(\y)\int_{\Omega}b_{2}(\x,\y)\nabla v(\x,\y)\cdot\nabla v(\x,\y)\dd\x \dd\y
              \\[4pt]
              & \leq\frac{\beta_{1}^{\max}}{\beta_{2}^{\min}}\, \B_2(v,v).
  \end{align*}
  The lower bound follows analogously. 
\end{proof}

The boundedness required in the above lemma for %the diffusion coefficients
\abrev{$b_i(\x,\y)$} holds for the \abrev{parametric diffusion coefficient} % fields
$a(\x,\y)$ (see~\Refx{eq:assumption:on:a}).
\abrev{In the next \rev{subsection} we show}
that \abrev{assumptions~\Refx{eq:assumption:on:a0},~\Refx{eq:assumption:on:ai-1},
which guarantee~\Refx{eq:assumption:on:a}, also yield}
boundedness of \abrev{\emph{truncated} expansions} of the \abrev{coefficient} $a(\x,\y)$.
%\abrev{That result} will \abrev{enable our main goal---the design and analysis
%of a new class of preconditioners for stochastic Galerkin matrices}.
\abrev{That result} will \abrev{enable our main goal---\rev{the analysis
of truncation preconditioners} for stochastic Galerkin matrices}.

%----------------section 3.2--------------------------------------------------------
\subsection{Spectral analysis} \label{sec:spectral:analysis}

%For clarity of exposition and
In analogy with (\ref{eq:assumption:on:ai-1}), we define
\begin{equation}
  \label{eq:tauess}
  \tau_0 := 0,\quad
  \tau_r:=\frac{1}{a_{0}^{\min}} \bigg\| \sum_{m=1}^{r} \abs{a_m} \bigg\|_{\infty},\ \ r \in \NN.
\end{equation}
Note that $(\tau_r)_{r \in \NN_0}$ is a monotonic increasing
sequence, which is bounded from above (cf. \Refx{eq:assumption:on:ai-1}),~i.e.,
\[
  0 \le \tau_r \leq \tau_{r+1} \le \tau<1,\quad r \in \NN_0.
\]
%
%We start with a result establishing the
First, we establish the boundedness of \emph{truncated} expansions %$a_r$
of the parametric diffusion coefficient~$a$.

\begin{lemma}\label{lem:asbound}
  \abrev{Assume that \Refx{eq:assumption:on:a0} and \Refx{eq:assumption:on:ai-1} hold.
  Let $r\in\NN_0$ and define $a_r:\Omega\times\bGamma\, \to \RR$ to be the
  finite~sum
  \begin{equation}
    a_r(\x,\y)\; \abrev{:=}\; a_{0}(\x)+\sum_{m=1}^ra_m(\x)y_m.
  \label{eq:form of a (param)-1}
  \end{equation}
  Then $a_r(\x,\y)$ is positive and bounded almost everywhere
  in~$\Omega\times\bGamma$.
}
\end{lemma}

\begin{proof}
  The case $r=0$ follows from the boundedness and positivity
  assumptions~\Refx{eq:assumption:on:a0} on
  $a_0(\x)$. Consider now $r\in\NN$. Since $\abs{y_m}\leq 1$, using the definition~\Refx{eq:tauess} of
  $\tau_r$ we obtain
  \[
    \left|a_r(\x,\y)-a_{0}(\x)\right|
    =\bigg|\sum_{m=1}^ra_m(\x)y_m\bigg|\leq\sum_{m=1}^r\left|a_m(\x)\right|\leq a_0^{\min}\tau_r
    \qquad
    \text{a.e. in $\Omega \times \bGamma$}.
  \]
  Hence,
\begin{equation} \label{eq:bound:ar:a0}
    a_{0}(\x)-a_0^{\min}\tau_r \leq a_r(\x,\y)\leq a_{0}(\x)+a_0^{\min}\tau_r
\end{equation}
  and using the boundedness of $a_0(\x)$ (see~\Refx{eq:assumption:on:a0}), we get
  \begin{equation}
   \eta_r^{\min}:=a_{0}^{\min}-a_0^{\min}\tau_r \leq a_r(\x,\y)\leq a_{0}^{\max}+a_0^{\min}\tau_r=:\eta_r^{\max}
   \quad
   \text{a.e. in $\Omega \times \bGamma$}.\label{eq:etadef}
 \end{equation}
The proof concludes by noting that $\eta_r^{\min} > 0$ since $\tau_r \leq \tau < 1$
 (cf. \Refx{eq:tauess} and \Refx{eq:assumption:on:ai-1}).
\end{proof}

Combining Lemmas~\ref{lem: equiv tool} and~\ref{lem:asbound}, we obtain the following result.

\begin{theorem}\label{thm: B equiv BN}
  Let $a:\Omega\times\bGamma \to \RR_{+}$ be a parametric diffusion coefficient
  given by \Refx{eq:affine} and let $\A : V \times V \to \RR$ be the associated bilinear form defined in~\Refx{eq:A-form}.
  Assume~\Refx{eq:assumption:on:a0} and~\Refx{eq:assumption:on:ai-1} hold.
  Let $a_r$ be given by~\Refx{eq:form of a (param)-1} and define the associated bilinear form as
  \[
    \A_r(v,w) := \int_{\boldsymbol{\Gamma}}p(\y)\int_{\Omega}a_r(\x,\y)\nabla v(\x,\y)\cdot\nabla w(\x,\y)\dd\x \dd\y.
  \]
  Then $\A$ and $\A_r$ are equivalent for any $r\in\NN_0$.
\end{theorem}

\begin{proof}
 By \Refx{eq:assumption:on:a}, the diffusion coefficients $a$ is bounded.
 Furthermore, by Lem\-ma~\ref{lem:asbound}, the coefficient $a_r$, $r \in \NN_0$, is bounded as well.
% $\Omega\times\bGamma$.
Consequently, by Lemma \ref{lem: equiv tool}, the bilinear forms are
 equivalent. In particular,
 $$\uptheta_r\A_r(v,v)\leq \A(v,v)\leq
 \Uptheta_r\A_r(v,v),$$
 where %, in the notation of Lemma \ref{lem:asbound},
 \begin{equation}
 \uptheta_r := \frac{a_{\min}}{\eta_r^{\max}}=\frac{(1-\tau)a_0^{\min}}{a_{0}^{\max}+a_0^{\min}\tau_r},\qquad
 \Uptheta_r := \frac{a_{\max}}{\eta_r^{\min}}=\frac{a_{0}^{\max}+a_0^{\min}\tau}{(1-\tau_r)a_0^{\min}}\label{eq:thetas}
\end{equation}
with $\eta_r^{\min}$ and $\eta_r^{\max}$ defined in~\Refx{eq:etadef}.
\end{proof}

\begin{remark}
  The equivalence of the bilinear form $\A_0$ associated with the pa\-ra\-me\-ter-free term $a_0(\x)$ in~\Refx{eq:affine}
  and the bilinear form $\A$ is well known (see, e.g., {\rm \cite[eq.~(2.5)]{BespalovPRR_19_CAS}}).
  Theorem~{\rm \ref{thm: B equiv BN}} extends this result to the case of arbitrary finite truncation
  of the affine-parametric coefficient $a(\x,\y)$.
\end{remark}

\begin{remark}
  The constants $\uptheta_r,\, \Uptheta_r$ in~\Refx{eq:thetas}
  depend on \abrev{$\tau$}, %on the bounds
  $a_0^{\min}$, $a_0^{\max}$ and indirectly on $r$, via $\tau_r$.
\end{remark}

Theorem~\ref{thm: B equiv BN}, combined with Proposition
\ref{prop:equiv}, indicates that the bilinear form $\A_r$ induces a family of
preconditioners for the SGFEM matrix $A$ in~\Refx{eq:A:u:f}.
This is made precise in the next theorem which is the main result of this~section.

\begin{theorem}\label{thm:truncprec}
  Let $\A,\, \A_r$ be defined as in Theorem {\rm \ref{thm: B equiv BN}} and
  assume \Refx{eq:assumption:on:a0} and~\Refx{eq:assumption:on:ai-1} hold.
  Let $\{\varphi_{ij}\}$ be the tensor-product basis
  for the finite dimensional space $\VhkM$ in~\Refx{eq:VhkM}
  and $A = \left[\A(\varphi_{ij},\varphi_{st})\right]$ be the associated SGFEM matrix.
  \abrev{For a fixed $r \in \NN_0$}
  define the preconditioner $P_r$ via
  $$P_r:=\left[\A_r(\varphi_{ij},\varphi_{st})\right].$$
  Then $P_r\sim A$ and the spectrum of $P_r^{-1}A$ satisfies
  \begin{equation}
  \Lambda(P_r^{-1}A)\subset[\uptheta_r,\Uptheta_r]\label{eq:bounds}
\end{equation}
with $\uptheta_r,\, \Uptheta_r$ defined in \Refx{eq:thetas}.
\end{theorem}

\begin{remark}
To obtain alternative spectral bounds under the assumption in~\Refx{eq:assumption:on:ai-1:KubPul},
we can define (cf.~\Refx{eq:tauess})
\[
  \ttau_0 := 0,\quad
  \ttau_r := \bigg\| a_{0}^{-1} \sum_{m=1}^{r} \abs{a_m} \bigg\|_{\infty},\ \ r \in \NN.
\]
Note that $(\ttau_r)_{r \in \NN_0}$ is a monotonic increasing sequence bounded from above by~$\ttau$.
Then, instead of~\Refx{eq:etadef} we obtain by using~\Refx{eq:assumption:on:a:KubPul}
\[
   \frac{1 - \ttau_r}{1 + \ttau}\, a(\x,\y) \leq
   (1-\ttau_r) a_{0}(\x) \leq
   a_r(\x,\y) \leq
   (1+\ttau_r) a_{0}(\x) \leq
   \frac{1 + \ttau_r}{1 - \ttau}\, a(\x,\y)\ \
   \text{a.e. in }\Omega\times\bGamma.
\]
This implies the equivalence of bilinear forms $\A,\; \A_0,\; \A_r$ as follows:
\[
   (1-\ttau_r) \A_0(v,v) \le \A_r(v,v) \le (1+\ttau_r) \A_0(v,v)\quad
   \forall\, v \in V
\]
and
\[
   \frac{1-\ttau}{1+\ttau_r}\, \A_r(v,v) \le \A(v,v) \le \frac{1+\ttau}{1-\ttau_r}\, \A_r(v,v)\quad
   \forall\, v \in V.
\]
Therefore, by Proposition~{\rm \ref{prop:equiv}}, the following spectral bounds hold: %spectrum of $P_r^{-1}A$ satisfies
\[
   \Lambda(P_0^{-1}P_r) \subset [ 1 - \ttau_r, 1 + \ttau_r ]
   \text{ \ and \ }
   \Lambda(P_r^{-1}A) \subset \bigg[ \frac{1-\ttau}{1+\ttau_r}, \frac{1+\ttau}{1-\ttau_r} \bigg].
\]
\end{remark}

\abrev{The preconditioners $P_r$, that we will refer to as \emph{truncation preconditioners},
are induced by the bilinear form $\A_r$, $r\in\NN_0$.
Therefore, by~\Refx{eq:AM}, there holds $P_r=A$ for all $r\geq M$. In practice, the values of
$r$ are expected to be small in order to allow for sparse
approximations of $A$ which can be efficiently implemented. Note also that $P_r$ can be written as a sum of Kronecker products, just
as was the case for the SGFEM matrix $A$ (cf.~\Refx{eq:decomp:A}):
\begin{equation}
  P_r=G_0\otimes K_0+\sum_{m=1}^rG_m\otimes K_m.\label{eq:Pr}
\end{equation}
}

\abrev{The result of Theorem~\ref{thm:truncprec}
indicates that the performance of the preconditioned Conjugate
Gradient method will be independent of discretization parameters,
but may depend on the choice of truncation parameter $r$. Since
$$\A(v,v)=\A_r(v,v)+\R(v,v)$$
with
$$\R(v,v):=\int_{\boldsymbol{\Gamma}}p(\y)\int_{\Omega}(a(\x,\y)-a_r(\x,\y))\nabla v(\x,\y)\cdot\nabla v(\x,\y)\dd\x \dd\y,$$
a smallness assumption of the form
\begin{equation}
  \label{eq:resbound}
  \abs{\R(v,v)}\leq \varepsilon_r \A(v,v),\quad 0<\varepsilon_r<1
\end{equation}
would allow for the following equivalence of $\A$ and $\A_r$
\begin{equation}
  \label{eq:equiv}
  (1-\varepsilon_r)\A(v,v)\leq \A_r(v,v)\leq (1+\varepsilon_r) \A(v,v).
\end{equation}
Since
$$a(\x,\y)-a_r(\x,\y)=\sum_{m=r+1}^\infty a_m(\x)y_m,$$
by the definition of $\R,$ assumption \Refx{eq:resbound} holds for
sufficiently large $r$. As a result,~\Refx{eq:equiv} implies the
\rev{following} eigenvalue bounds
$$\Lambda(P_r^{-1}A)\subset\left[\frac{1}{1+\varepsilon_r},\frac{1}{1-\varepsilon_r}\right].$$
This suggests that the closer $a_r$ approximates $a$, the tighter the
preconditioned spectrum will be clustered around unity. We will
investigate this conclusion numerically in section~\ref{sec:numerics}.
}

%-----------------section 4-------------------------------------------------------
\section{Modified truncation preconditioners} \label{sec:modified:precond}

Any practical implementation of a preconditioner requires an efficient
technique for applying the action of its inverse on a given
vector. Standard approaches include constructing sparse
factorizations, or employing multigrid or multilevel techniques;
domain decomposition methods represent yet another
approach. The potential for parallelism could also be a deciding
factor in the choice of solution method.

The preconditioner $P_r$ \rev{introduced} in the previous section is block-sparse,
with sparsity deteriorating with increasing $r$
(cf. Theorem~\ref{thm:G:matrices}). In the case when $r=1$, the
structure can be shown to be block-tridiagonal under a certain
permutation---this is not an ideal situation, as it requires
additional techniques to ensure an efficient application of the
preconditioner. For this reason, we replace $P_r$ with its
corresponding symmetric block Gauss--Seidel (SBGS) approximation:
\begin{equation}
  \abrev{\tilde{P}_r :=} \left(G_0\otimes K_0+\sum_{m=1}^rL_m\otimes
  K_m\right)(G_0\otimes K_0)^{-1}\left(G_0\otimes K_0+\sum_{m=1}^rL_m^T\otimes
  K_m\right),\label{eq:prt}
\end{equation}
where $L_m+L_m^T=G_m$.
\abrev{The matrix $\tilde{P}_r$ thus} %This approximation
represents a sparse approximation to $P_r$
involving block-triangular and block-diagonal matrices. In the remainder of this section we prove that \abrev{$P_r \sim \tilde{P}_r$}
\abrev{and} provide complexity considerations, including
a discussion of implementation. %\todo{potential} practical algorithms.

%--------------section 4.1-------------------------------------------------------------
\subsection{Analysis of SBGS approximation of $P_r$} \label{sec:SBGS:analysis}

\rev{In this subsection we assume that the ordering of multi-indices in the index set $\I_{k}^{M}$ is such that
the matrices $L_m$ in~\Refx{eq:prt} have at most one nonzero entry per row and per column
(this property holds, e.g., for lexicographic or anti-lexicographic ordering as well as for 
ascending or descending ordering by the total degree of the associated complete polynomials in~$\SkM$).
Let us define}
\begin{equation}
  S_r \;\abrev{:=}\; \sum_{m=1}^rL_m\otimes K_m,\quad D_0 \;\abrev{:=}\; G_0\otimes K_0,
  \label{eq:Sr:D0}
\end{equation}
so that
\begin{equation}
  P_r \;\abrev{\stackrel{\Refx{eq:Pr}}{=}} D_0+S_r+S_r^T
  \label{eq:Pr:mod}
\end{equation}
and
$$\tilde{P}_r \;\abrev{\stackrel{\Refx{eq:prt}}{=}}\; (D_0+S_r)D_0^{-1}(D_0+S_r^T)=P_r+S_r D_0^{-1}S_r^T.$$
Our spectral analysis %will
focuses on deriving bounds for the generalized Rayleigh~quotient
\begin{equation}
  \frac{\bv^T\tilde{P}_r\bv}{\bv^TP_r\bv}=1+\frac{\bv^T S_r
  D_0^{-1}S_r^T\bv}{\bv^T(D_0+S_r+S_r^T)\bv},\qquad \bv\in\RR^{\Nx\Ny}\setminus\seq{\bf 0}.
  \label{eq:Rayleigh}
\end{equation}
Since the lower bound is 1, we restrict our attention to deriving an
upper bound for the second term on the right-hand side of~\Refx{eq:Rayleigh}, which we write using the change
of variable $\w=D_0^{1/2}\bv$ as
\begin{equation*}
  \rho(\w)\;\abrev{:=}\; \frac{\w^T\tilde{S}_r\tilde{S}_r^T\w}{\w^T(I+\tilde{S}_r+\tilde{S}_r^T)\w}.
%  \label{eq:rho:def}
\end{equation*}
Here,
\begin{equation}
   \tilde{S}_r\;\abrev{:=}\; D_0^{-1/2}S_rD_0^{-1/2}=\sum_{m=1}^rL_m\otimes\tilde{K}_m
   \label{eq:Srt}
\end{equation}
with $\tilde{K}_m\;\abrev{:=}\; K_0^{-1/2}K_mK_0^{-1/2}$, using the fact that $G_0=I_{\Ny}$.
Hence,
\begin{equation}
    \rho(\w)\leq\max_{\w\neq{\bf
    0}}\frac{\w^T\tilde{S}_r\tilde{S}_r^T\w}{\w^T\w}\cdot\max_{\w\neq{\bf 0}}\frac{\w^T\w}{\w^T(I+\tilde{S}_r+\tilde{S}_r^T)\w}=
    \frac{\sigma_{\max}^2(\tilde{S}_r)}{\lambda_{\min}(I+\tilde{S}_r+\tilde{S}_r^T)},
    \label{eq:rho:bound}
  \end{equation}
  where $\sigma_{\max}(\cdot)$ and $\lambda_{\min}(\cdot)$ denote, respectively,
the largest singular value and the smallest eigenvalue of a~matrix.
In the next lemma, we provide bounds for $\sigma_{\max}(\tilde{S}_r)$ and $\lambda_{\min}(I+\tilde{S}_r+\tilde{S}_r^T)$ in order to
conclude the derivation of the upper bound on~$\rho$.

\begin{lemma} \label{lm:aux:bounds}
\abrev{Suppose that \Refx{eq:assumption:on:a0} and \Refx{eq:assumption:on:ai-1} hold
\rev{and $\tau_r$ is defined in~\Refx{eq:tauess}}.
%and let $\eta_r^{\min}, \eta_r^{\max}$ be defined\break in~\Refx{eq:etadef}.
}
Let $\tilde{S}_r$ be defined \abrev{by~\Refx{eq:Srt}}. Then
\begin{equation}
     \lambda_{\min}(I+\tilde{S}_r+\tilde{S}_r^T) \geq \rev{1 - \tau_r} % \frac{\eta_r^{\min}}{a_0^{\max}}
     \label{eq:bound:lambda}
\end{equation}
and
  \begin{equation}
     \sigma_{\max}(\tilde{S}_r)\leq
     \abrev{\frac{1}{a_0^{\min}}\sum_{m=1}^r \norm{a_m}{\infty}}.
     \label{eq:bound:sigma}
  \end{equation}
\end{lemma}

\begin{proof}
Since
\[
   I+\tilde{S}_r+\tilde{S}_r^T =
   I+\sum_{m=1}^rG_m\otimes \tilde{K}_m =
   D_0^{-1/2} \bigg( D_0+\sum_{m=1}^rG_m\otimes K_m \bigg) D_0^{-1/2}
   \,{=}\, D_0^{-1/2} P_r D_0^{-1/2},
\]
the eigenvalues of $I+\tilde{S}_r+\tilde{S}_r^T$ are the eigenvalues
of $D_0^{-1}P_r =P_0^{-1}P_r$.
\rev{To find the bounds on the spectrum of $P_0^{-1}P_r$, recall the inequalities in~\Refx{eq:bound:ar:a0}
and the lower bound for $a_0(\x)$ in~\Refx{eq:assumption:on:a0}, which together imply~that
\[
    (1-\tau_r) a_{0}(\x) \leq a_r(\x,\y) \leq (1 + \tau_r) a_{0}(\x)\quad\text{a.e. in }\Omega\times\bGamma.
\]
Hence, the bilinear forms $\A_0$ and $\A_r$ are equivalent and by Proposition~\ref{prop:equiv} there~holds
\begin{equation} \label{eq:spectrum:0:r}
   \Lambda(P_0^{-1}P_r) \subset \left[ 1 - \tau_r, 1 + \tau_r \right].
\end{equation}
}
%
%Since the coefficients $a_0,a_r$ are positive and bounded (cf. Lemma \ref{lem:asbound}), the
%bilinear forms inducing $P_0,P_r$ are equivalent by Lemma \ref{lem: equiv tool}
%and therefore $P_0\sim P_r$ by Proposition \ref{prop:equiv}.  In
%particular, %using the notation of Lemma \ref{lem:asbound},
%we obtain
%$$\Lambda(P_0^{-1}P_r)\subset\left[\frac{\eta_r^{\min}}{a_0^{\max}},\frac{\eta_r^{\max}}{a_0^{\min}}\right].$$
%
This proves~\Refx{eq:bound:lambda}.

On the other hand, since $\sigma_{\max}$ \abrev{defines} a norm, we use the triangle inequality to~estimate
\begin{equation}
     \sigma_{\max}(\tilde{S}_r) \leq
     \sum_{m=1}^r\sigma_{\max}(L_m\otimes \tilde{K}_m) \leq
     \sum_{m=1}^r\sigma_{\max}(L_m)\sigma_{\max}(\tilde{K}_m).
     \label{eq:bound:aux1}
\end{equation}
Now, $\sigma^2_{\max}(L_m)=\lambda_{\max}(L_mL^T_m)$; since $L_mL_m^T$
is diagonal for \abrev{every} $m$
\abrev{(due to $L_m$ having \rev{at most one nonzero entry per row and per column})},
it follows that for all $m$
\begin{equation}
     \sigma_{\max}(L_m)=\max_{i,j} \abrev{[G_m]}_{ij} \le \max_{k} c_k^m \leq 1
     \label{eq:bound:aux2}
\end{equation}
\abrev{with $c_k^m$ %arising in the three-term recurrence \Refx{eq:3term}
being bounded by 1 (cf. Theorem~\ref{thm:G:matrices} and inequalities~\Refx{eq:cmj}).}
Finally, since the eigenvalues of $\tilde{K}_m$ are the eigenvalues of
$K_0^{-1}K_m$, we find~\abrev{that}
\[
   \sigma_{\max}(\tilde{K}_m)=\max_{i}\abs{\lambda_i(K_0^{-1}K_m)}\leq
   \frac{\norm{a_m}{\infty}}{a_0^{\min}}
\]
\abrev{and then inequality~\Refx{eq:bound:sigma} follows from~\Refx{eq:bound:aux1} and~\Refx{eq:bound:aux2}.
This finishes the~proof.}
%Hence,
%$$\sigma_{\max}(\tilde{S}_r)\leq \frac{c}{a_0^{\min}}\sum_{m=1}^r \norm{a_m}{\infty}$$
\end{proof}

We summarize our discussion in the following result.

\begin{proposition} \label{prop:tildebounds}
Suppose that \Refx{eq:assumption:on:a0} and \Refx{eq:assumption:on:ai-1} hold
\rev{and $\tau_r$ is defined in~\Refx{eq:tauess}}.
% and let $\eta_r^{\min}$ be defined\break in~\Refx{eq:etadef}.
  Let $P_r$ be defined in \Refx{eq:Pr} and let $\tilde{P}_r$ be its
  symmetric block Gauss--Seidel approximation~\Refx{eq:prt}.
  Then $\tilde{P}_r\sim \abrev{P_r}$ and the spectrum of \abrev{$P_r^{-1} \tilde{P}_r$} satisfies
  \begin{equation}
    \label{eq:tildebounds}
%    1\leq \Lambda(P_r^{-1}\tilde{P}_r)\leq 1+\delta_r,
    \abrev{\Lambda(P_r^{-1}\tilde{P}_r) \subset [1,\, 1+\delta_r]},
  \end{equation}
  where
  \begin{equation}
  \delta_r := \rev{\frac{1}{1-\tau_r}} % \frac{a_0^{\max}}{\eta_r^{\min}}
                 \bigg( \frac{1}{a_0^{\min}} \sum_{m=1}^r \norm{a_m}{\infty} \bigg)^2.
   \label{eq:deltar}
\end{equation}
\end{proposition}

\abrev{The proof of the upper bound in~\Refx{eq:tildebounds}
is completed by substituting the estimates~\Refx{eq:bound:lambda}, \Refx{eq:bound:sigma} into % inequality~
\Refx{eq:rho:bound} and then using the resulting bound for $\rho(\w)$ in~\Refx{eq:Rayleigh}.}

The following \abrev{result} is a straightforward consequence of
\abrev{Theorem~\ref{thm:truncprec}, Proposition~\ref{prop:tildebounds} and the}
transitivity of spectral equivalence.

\begin{theorem}
  Let $\tilde{P}_r$ be the symmetric block Gauss--Seidel approximation \Refx{eq:prt}
  to $P_r$. Let
  $\uptheta_r, \Uptheta_r$ be defined in \Refx{eq:thetas} and let
  $\delta_r$ be defined in \Refx{eq:deltar}. Then $\tilde{P}_r\sim A$ and the spectrum of
  $\tilde{P}_r^{-1}A$ satisfies
  \begin{equation}
    \label{eq:tildebounds1}
%    \frac{\uptheta_r}{1+\delta_r}\leq \Lambda(\tilde{P}_r^{-1}A)\leq \Uptheta_r.
    \abrev{\Lambda(\tilde{P}_r^{-1}A) \subset \bigg[\frac{\uptheta_r}{1+\delta_r},\, \Uptheta_r\bigg]}.
  \end{equation}
\end{theorem}

%-----------section 4.2-------------------------------------------------------------
\subsection{Implementation. Complexity considerations} \label{sec:implement}

Our proposed solution method for solving the linear system
\Refx{eq:sGFEM} is
the preconditioned Conjugate Gradient (PCG) method, for which the main
computational effort at each step comprises a matrix-vector product with the matrix
$A$ and the solution of a linear system with the preconditioning
matrix.
Since the main computational cost is associated with the
latter operation, we discuss this in detail.
We will denote by $\fl(operation)$ the complexity, i.e., number of flops required to perform \abrev{an} $operation$.
The number of nonzeros of a matrix will be denoted by $\nnz(\cdot)$.
%
%We will denote by $\fl(operation)$ the number of flops required to perform \abrev{an} $operation$;
%we will refer to $\fl(\cdot)$ as the sequential complexity.
%The parallel complexity will be referred to by $\flp(\cdot)$,
%while the number of nonzeros of a matrix will be denoted by $\nnz(\cdot)$.

As indicated previously, the action of the inverse of $P_r$ needs to
be approximated, due to its sparse (but non-diagonal) block structure.
%
%This can be achieved in a number of ways. We consider two options:
%\begin{itemize}
%\item replace $P_r$ with its SBGS approximation $\tilde{P}_r$;
%\item solve a linear system with $P_r$ iteratively (e.g., using PCG)
%  with a suitable preconditioner. This would add an inner iteration to
%  the existing outer PCG iteration.
%\end{itemize}
%
We achieve this by replacing $P_r$ with its SBGS approximation $\tilde{P}_r$.
%
%Consider first 
Let us consider the implementation of the action of $\tilde{P}_r^{-1}$ onto a given vector~$\bv$;
for general $r \in \NN_0$, this can be achieved as follows
(see~\Refx{eq:Sr:D0}--\Refx{eq:Pr:mod} for the definitions of the
  respective matrices):
\begin{enumerate}
\item solve $(D_0+S_r)\w=\bv$;
\item solve $(D_0+S_r^T)\z=D_0\w$.
\end{enumerate}
Both of the above steps involve the solution of a block-triangular
system, with the main computational cost arising from solving linear
systems with the diagonal blocks $K_0$. Specifically, since $P_r$ has
at most $2r+1$ nonzero block matrices per row, we find 
$$\fl(\tilde{P}_r^{-1}\bv)\approx (2rN_\y)
\nnz(K_0)+2N_\y\fl(K_0^{-1}\b),$$
for some vector $\b$ of size $N_\x$.
% In general, the above implementation does not have an obvious parallel version for a generic $r$.
%However, in our numerical experiments, low values of $r$ appear to yield
%optimal preconditioners. For this reason, we consider some special
%cases below.
%
Below we consider two special cases.

\subsubsection{Special case: $r=0$}
The preconditioner $P_0$ is the mean-based preconditioner introduced in \cite{ghanemkruger96}.
%The complexities associated with the action of the inverse on a given vector are
%$$\fl({P}_0^{-1}\bv)=N_\y\fl(K_0^{-1}\b),\qquad \flp({P}_0^{-1}\bv)=\fl(K_0^{-1}\b).$$
%
The complexity associated with the action of the inverse on a given vector is
\[
   \fl({P}_0^{-1}\bv)=N_\y\fl(K_0^{-1}\b).
\]

\subsubsection{Special case: $r=1$}
The structure of $P_r$ simplifies
greatly when $r=1$.~In particular, $G_1$ has a block-diagonal
structure under a certain permutation~\cite[\S9.5]{lord14}:
\[
  G_1 = \diag( T_{k+1},\,T_k,\ldots,T_k,\ldots,T_1,\ldots,T_1),
\]
%\[
%  G_1=
%  \begin{bmatrix}
%    T_{k+1} & & & & & & &\\
%    & T_k  & & & & & &\\
%    & &  \ddots& & & & &\\
%    & & & T_k & & & &\\
%    & & & & \ddots & & &\\
%    & & & & & T_1& &\\
%    & & & & & & \ddots&\\
%    & & & & & & &T_1\\
%  \end{bmatrix}
%  \]
where $T_j\in\RR^{j\times j}$ $(j=1,\ldots,k+1)$ are tridiagonal, with zero main diagonal. Note in
particular that $T_1=0$. As a result, $P_1$ will have a block-diagonal
structure, where each diagonal block is a block-tridiagonal matrix of
size $jN_\x$. Specifically,
$$P_1=\bigoplus_{j=1}^{k+1}\bigoplus_{i=1}^{n_j}\left(I\otimes
  K_0+T_j\otimes K_1\right),$$
where, assuming $M>1$,
$$n_j={k+M-j-1 \choose M-2}.$$
Given this structure, the % sequential
complexity for the implementation of the
action of $\tilde{P}_1^{-1}$ on a given vector $\bv$ is
$$\fl(\tilde{P}_1^{-1}\bv)\approx
2\left(\sum_{j=2}^{k+1}jn_j\right)\left(\nnz(K_0)+\fl(K_0^{-1}\b)\right)+n_1 \fl(K_0^{-1}\b).$$
Under the assumption that $\fl(K_0^{-1}\b)$ dominates the computation, we deduce that
the implementation of $\tilde{P}_1$ is at most twice as expensive as the %sequential
implementation of $P_0$.

%\todo{With regard to the parallel
%complexity, we noted previously that, in general,
%there is no obvious parallel implementation of $\tilde{P}_r$. 
%However, the above block structure allows for the parallel
%implementation of the actual preconditioner ${P}_1$, due to the tridiagonal structure
%present in the blocks $T_j$. More precisely, a parallel strategy such as that
%presented in \cite{vandervorst87} yields the following parallel complexity
%$$\flp({P}_1^{-1}\bv)\approx
%2\left(\nnz(K_0)+\fl(K_0^{-1}\b)\right).$$
%This indicates that $P_1$ can be implemented efficiently in parallel,
%with a parallel complexity roughly twice that of $P_0$.
%}

\subsubsection{Kronecker preconditioner}
We end this section with a
discussion of the complexity required for an implementation of the
Kronecker preconditioner \cite{Ullmann10}. Since
$$P_\otimes=G\otimes K_0=(G\otimes I_{\Nx})(I_{\Ny}\otimes K_0)=(G\otimes
I_{\Nx})P_0,$$
% the complexities are
%\[
%   \fl({P}_\otimes^{-1}\bv) = N_\y \fl(K_0^{-1}\b) + N_{\x} \fl(G^{-1}\d),\qquad
%   \flp({P}_\otimes^{-1}\bv) = \fl(K_0^{-1}\b)+\fl(G^{-1}\d),
%\]
%
the complexity is given by
\[
   \fl({P}_\otimes^{-1}\bv) = N_\y \fl(K_0^{-1}\b) + N_{\x} \fl(G^{-1}\d)
\]
for some vector $\d$ of size $\Ny$. Thus, the complexity exceeds that
of $P_0$ by a computational cost dependent on the sparsity of $G$. In
particular, it was estimated in \cite{Ullmann10} that this additional
cost would amount to $\fl(G^{-1}\d) \sim O((2M+1)^2)$ operations,
excluding the cost of performing a Cholesky factorization of $G$.

%Thus, under the assumption that $\fl(K_0^{-1}\b)\gg M^2$, a generic comparison of $P_0, P_1, P_\otimes$ would yield the
%following ranking of parallel complexities
%$$\flp({P}_0^{-1}\bv)<\flp({P}_\otimes^{-1}\bv)<2
%\flp({P}_0^{-1}\bv)\approx \flp({P}_1^{-1}\bv).$$

%-----section5-------------------------------------------------------------------
\section{Numerical experiments} \label{sec:numerics}
\abrev{
In this section, we investigate the effectiveness of
\rev{the preconditioning strategies considered in~\S\S\ref{sec:preconditioners}--\ref{sec:modified:precond}}.
% our proposed ... for computing stochastic Galerkin approximations.
In particular, we verify the theoretical optimality of truncation preconditioners $P_r$ and $\tilde P_r$ ($r \ge 1$)
with respect to discretization parameters
and compare their performance with that of
the mean-based preconditioner $P_0$ and
the Kronecker product preconditioner $P_{\otimes}$ defined in \Refx{eq:kron}
(see \cite{Ullmann10} for details and analysis).
}

We chose to use test problems
satisfying the descriptions and assumptions in this paper, as well as
problems outside the theoretical framework. Thus, we solved model
  problem~\Refx{eq:strong:form} using the following choices of
  parametric diffusion coefficient:
\begin{itemize}
  \item $a(\x,\y)$  has the affine representation
    \Refx{eq:affine}, with the coefficients $a_m(\x)$ satisfying
    \Refx{eq:assumption:on:a0} and \Refx{eq:assumption:on:ai-1};
  \item $a(\x,\y)$ is a lognormal diffusion coefficient, i.e.,
    $a(\x,\y)=\exp(b(\x,\y))$, where $b(\x,\y)$ is assumed to have the affine representation
    \Refx{eq:affine} but with unbounded parameters $y_m$
    \rev{(we note that this choice of coefficient $a$ is not covered by our theoretical~analysis)}.
%    ; indeed, even the generalization of our class of preconditioners is not straightforward.
%    We discuss a possible approach in~\S\ref{sec:numerics:non-affine}.
\end{itemize}

In all our tests, we chose $\Omega=(0,1)^2$ and $f(\x)=1$. 
We used the MATLAB toolbox
S-IFISS~\cite{SIFISS} to generate SGFEM discretizations of our model
problem for a range of discretization parameters. We used uniform
subdivisions of $\Omega$ into square elements of edge length~$h$, with
$h$ ranging between $2^{-3}$ to $2^{-7}$.
The discretization parameters $k,M$ had ranges $1\leq k\leq 6$, $1\leq M\leq 8$.
% This resulted in linear systems with sizes of order exceeding $O(10^7)$.
%An illustration of the growth in problem size with the discretization
%parameters $h,k,M$ is included in Figure \ref{fig:dofs}.
%
We solved the resulting linear systems using the preconditioned CG method with
tolerance $tol = 10^{-6}$ and zero initial guess.

%----------------------------------
% section 5.1
%----------------------------------------------------------------------------------------------
\subsection{Test problem 1: affine-parametric diffusion coefficient} \label{sec:numerics:affine}

\abrev{For this test problem, the diffusion coefficient $a(\x,\y)$  had
  the affine-parametric form~\Refx{eq:affine}, with the coefficients
$a_m(\x)$ in the expansion chosen such that they exhibit either slow or fast decay.
As indicated at the end of section~\ref{sec:preconditioners}, this is expected to affect the performance of the
truncation preconditioners $P_r$. In particular, our numerical
experiments will highlight the dependence on the truncation parameter $r$.
}

Following~\cite[Section~11.1]{MR3154028}, we select the expansion coefficients $a_m(\x)$ ($m \in \NN_0$) in~\Refx{eq:affine}
to represent planar Fourier modes of increasing total order,~i.e.,
%\begin{example}	\label{exa:ex5}
\begin{equation}
	a_{0}(\x) \,{=}\, 1,\ 
	a_{m}(\x) \,{=}\, \bar\alpha m^{-\tilde\sigma} \cos\big(2\pi\beta_{1}(m)x_{1}\big) \cos\big(2\pi\beta_{2}(m)x_{2}\big),
	\ \x \,{=}\, (x_{1},x_{2}) \,{\in}\, \Omega.
	\label{eq:ex5}
\end{equation}
Here, $\tilde{\sigma}>1$, $0<\bar{\alpha}<1/\zeta(\tilde{\sigma})$, where $\zeta$ denotes the Riemann zeta function,
and $\beta_1,\,\beta_2$ are given by
\begin{align*}
	\beta_{1}(m)=m-k(m)\left(k(m)+1\right)/2,\qquad \beta_{2}(m)=k(m)-\beta_{1}(m)
\end{align*}
with $k(m)=\left\lfloor -1/2+\sqrt{1/4+2m}\right\rfloor$.
Furthermore, we assume the parameters $y_m$ ($m \in \NN$) in~\Refx{eq:affine} to be the images
 of independent uniformly distributed mean-zero random variables on $\Gamma_m = [-1,1]$.
 In this case, $p_m(y_m) = 1/2$, $y_m \in \Gamma_m$, and 
 the orthonormal polynomial basis of $L^2_{\pi_m}(\Gamma_m)$ consists of scaled Legendre polynomials.
Note that with these settings, conditions~\Refx{eq:assumption:on:a0} and~\Refx{eq:assumption:on:ai-1}
are satisfied with constants $a_{0}^{\min}=a_{0}^{\max}=1$ and
$\tau \le \bar{\alpha} \zeta(\tilde{\sigma})$, respectively.

\abrev{The choice $\tilde{\sigma}=2$ in \Refx{eq:ex5} yields coefficients
$a_m$ with slow decay of the amplitudes $\bar\alpha m^{-\tilde\sigma}$, while the fast decay corresponds to the choice $\tilde{\sigma}=4$.
In each case, we select the factor $\bar\alpha$ such that $\bar{\alpha} \zeta(\tilde{\sigma}) = 0.9999$,
which gives $\bar{\alpha} \approx 0.6079$ for $\tilde{\sigma}=2$ and
$\bar{\alpha} \approx 0.9239$ for $\tilde{\sigma}=4$.}
The magnitudes of the expansion coefficients $a_m$ for increasing $m$
for slow and fast decay are displayed in
Table~\ref{tab:magnitude:decay}.

%----------------------------
\begin{table}[!b]
	\centering{}%
{\small
	\begin{tabular}{c|ccccccc}
		$m$ & 0 & 1 & 2 & 3 & 4 & 5 & 6\\
          \hline 
          {\Large\strut}
          slow decay ($\tilde{\sigma}=2)$ & {1.0000} & {0.6079} & {0.1520} & {0.0675} & {0.0380} & {0.0243} & {0.0169}\\
          \hline 
          {\Large\strut}
          fast decay ($\tilde{\sigma}=4$) & {1.0000} & {0.9239} & {0.0577} & {0.0114} & {0.0036} & {0.0015} & {0.0007}\\
	\end{tabular}
	\caption{Magnitudes $\left\Vert a_{m}\right\Vert _{\infty}$ of expansion coefficients~\Refx{eq:ex5} for test problem~1.
		% for different values of $\tilde{\sigma}$.
		}
	\label{tab:magnitude:decay}
}
\end{table}
%----------------------------

While the magnitudes of $a_m$ ($m = 1,2,\ldots$)
(and hence, the importance of the corresponding parameters $y_m$, $m=1,2,\ldots$)
decay significantly faster for $\tilde\sigma = 4$
(e.g., $\|a_1\|_{\infty} \approx 16 \|a_2\|_{\infty}$ for $\tilde\sigma = 4$, whereas
$\|a_1\|_{\infty} \approx 4 \|a_2\|_{\infty}$ for $\tilde\sigma = 2$),
we observe that the magnitude of $a_1$ is much closer
to the magnitude of the mean field $a_0$ in the case of fast decay
than in the case of slow decay.
This suggests that for fast decay there holds $a(\x,\y)\approx
a_1(\x,\y)$, which in turn implies that equivalence \Refx{eq:equiv} may
hold for $r=1$ and with a small $\varepsilon_1$. We therefore expect the performance of $P_1$ to
be superior in the fast decay case. This is indeed confirmed by the
iteration counts in Table~\ref{tab:ex5:ideal:precond}. The results
also confirm the optimal performance of $P_r$ with respect to $k$.

A similar behavior can be observed also for the case where the %ideal
preconditioners $P_r$ are replaced by their symmetric block Gauss--Seidel approximations
$\tilde{P}_r$. Table~\ref{tab:ex5:modified:precond} diplays the
corresponding iteration counts for a range of $r$, as well as for the
mean-based preconditioner $P_0$ and Kronecker preconditioner
$P_{\otimes}$, for both fast and slow decay cases.
%\todo{However, we note that the use of $\tilde{P}_r$ introduces a very mild dependence on $k$.}

%----------------------------
\begin{table}[!t]
\begin{center}
{\small
	\begin{tabular}{r| *7{c} | *7{c}}
	\multirow{2}{*}{$k$}& \multicolumn{7}{c|}{fast decay} & \multicolumn{7}{c}{slow decay} \\
	\cline{2-15}
	{\large\strut} & $P_{0}$ & $P_{1}$ & $P_{2}$ & $P_{3}$ & $P_{4}$ & $P_{5}$ & $P_{6}$   &
	          $P_{0}$ & $P_{1}$ & $P_{2}$ & $P_{3}$ & $P_{4}$ & $P_{5}$ & $P_{6}$\\
	\hline 
	1 & 13 & 4 & 3 & 3 & 2 & 2 & 2   &   10 & 6 & 4 & 4 & 4 & 3 & 3\\
	2 & 16 & 5 & 4 & 3 & 3 & 2 & 2   &   12 & 7 & 5 & 5 & 4 & 4 & 3\\
	3 & 21 & 6 & 4 & 3 & 3 & 2 & 2   &   14 & 7 & 6 & 5 & 4 & 4 & 4\\
	4 & 24 & 6 & 4 & 3 & 3 & 3 & 2   &   15 & 8 & 6 & 5 & 4 & 4 & 4\\
	\end{tabular}
	\caption{PCG iterations counts for test problem~1; % expansion coefficients \Refx{eq:ex5};
	$h=2^{-4}, M=8$.}
	\label{tab:ex5:ideal:precond}
}
\end{center}
\end{table}
%----------------------------

%----------------------------
\begin{table}[!b]
\setlength\tabcolsep{5.7pt} 
\begin{center}
{\small
	\begin{tabular}{r| *8{c} | *8{c}}
	\multirow{2}{*}{$k$} & \multicolumn{8}{c|}{fast decay} & \multicolumn{8}{c}{slow decay} \\
	\cline{2-17}
	{\Large\strut}& $P_{\otimes}$ & $P_{0}$ & $\tilde P_{1}$ & $\tilde P_{2}$ & $\tilde P_{3}$ & $\tilde P_{4}$ & $\tilde P_{5}$ & $\tilde P_{6}$   &
	                        $P_{\otimes}$ & $P_{0}$ & $\tilde P_{1}$ & $\tilde P_{2}$ & $\tilde P_{3}$ & $\tilde P_{4}$ & $\tilde P_{5}$ & $\tilde P_{6}$\\
	\hline
	1 & 12 & 13 & 7 & 6 & 6 & 6 & 6 & 6		   &   9 & 10 & 6 & 5 & 5 & 5 & 5 & 5\\
		2 & 16 & 16 & 8 & 7 & 7 & 7 & 7 & 7		   & 12 & 12 & 7 & 6 & 6 & 6 & 5 & 5\\
		3 & 20 & 21 & 9 & 9 & 8 & 8 & 8 & 8		   & 14 & 14 & 8 & 7 & 6 & 6 & 6 & 6\\
		4 & 24 & 24 & 10 & 9 & 9 & 9 & 9 & 9		   & 15 & 15 & 9 & 7 & 7 & 6 & 6 & 6\\
		5 & 26 & 27 & 11 & 10 & 10 & 10 & 10 & 10	   & 16 & 16 & 9 & 7 & 7 & 7 & 6 & 6\\
		6 & 29 & 29 & 12 & 11 & 11 & 11 & 11 & 11	   & 17 & 17 & 10 & 8 & 7 & 7 & 7 & 7\\
	\end{tabular}
        \caption{PCG iterations counts for test problem~1; % expansion coefficients \Refx{eq:ex5};
          $h=2^{-4}, M=8$.
        }
	\label{tab:ex5:modified:precond}
        }
\end{center}
\end{table}
% ----------------------------

\abrev{The results in Tables~\ref{tab:ex5:ideal:precond}
  and~\ref{tab:ex5:modified:precond} indicate that
the iteration counts corresponding to the approximations $\tilde{P}_r$
of the preconditioners $P_r$ are higher. This is expected given the
theoretical deterioration of the spectral bounds in
\Refx{eq:tildebounds1} as compared to the bounds in \Refx{eq:bounds}}.
However, all truncation preconditioners require fewer iterations than their mean-based and Kronecker product counterparts.
This improvement in the iteration counts is more pronounced
in the case of fast decay than in the case of slow decay, which is consistent, in particular,
with how the magnitudes of expansion coefficients $a_m$ change with $m$
(see Table~\ref{tab:magnitude:decay} and the associated discussion above).
For example, the numbers of iterations for the truncation preconditioner $P_1$ (resp., $\tilde P_1$)
are less than those for the mean-based preconditioner $P_0$ by factors of 3 to 4
(resp., by factors of about 2 to 2.5) in the case of fast decay.
This is because in this case, the expansion coefficient $a_1$ has approximately the same magnitude as the mean field.
In the case of slow decay, however, both $P_1$ and $\tilde P_1$ outperform $P_0$
in terms of the number of iterations only by a factor between 1.5 and 2.
%It is worth recalling here that the computational cost for truncation
%preconditioners $P_1$ and $\tilde P_1$ is about twice the cost of
%$P_0$ (see~\S\ref{sec:implement}); thus, in terms of the overall computational complexity,
%the truncation preconditioners perform at least the same as the mean-based one;
%
It is worth recalling here that the computational cost for the truncation preconditioner $\tilde P_1$
is about twice the cost of the mean-based preconditioner $P_0$ (see~\S\ref{sec:implement});
thus, in terms of the overall computational complexity, $\tilde P_1$ performs at least the same as $P_0$;
in the fast decay case, the overall computational cost for modified truncation preconditioners
is significantly lower than that for $P_0$.

If more expansion coefficients are retained in $P_r$ ($r \ge 2$),
then the iteration counts naturally (and consistently) decrease;
in particular, they decrease quicker in the case of fast decay of coefficient amplitudes
than in the case of slow decay of the amplitudes; see Table~\ref{tab:ex5:ideal:precond}.
This is again in agreement with what one might expect and reflects different decay patterns
of $\|a_m\|_\infty$ in each of these cases, as shown in Table~\ref{tab:magnitude:decay}.
However, when applying the corresponding modified truncation preconditioners $\tilde P_r$ ($r \ge 2$)
and increasing the number $r$ of retained expansion coefficients, the iteration counts decrease very slowly
(in the case of fast decay they even stagnate for $r \ge 2$ in most cases);
see Table~\ref{tab:ex5:modified:precond}. \abrev{This indicates that
  no significant improvement is obtained by including additional terms
  $a_m$ in the definition of $P_r$.}

The above set of experiments demonstrates that the modified truncation preconditioners $\tilde P_r$
provide sufficiently accurate approximations of stochastic Galerkin matrices and thus
can be used as effective practical preconditioners for linear systems arising from SGFEM approximations
of the model problem~\Refx{eq:strong:form} with \emph{affine-parametric} representation of the diffusion coefficient.
Depending on the decay of magnitudes of the expansion coefficients,
one may choose larger values of $r$ to improve the efficiency of the
solver (e.g., in the case of slow decay). However, in most cases, we recommend to choose $r=1$ or $r=2$.

We end the discussion of our first test problem with a numerical
confirmation of optimality of the modified truncation preconditioners %$\tilde P_1$
with respect to discretization parameters $h$ and $M$.
We consider again the cases of fast ($\tilde{\sigma}=4$) and slow ($\tilde{\sigma}=2$)
decay of coefficient amplitudes $\bar\alpha m^{-\tilde\sigma}$ in~\Refx{eq:ex5} and
%focus on the fast decay case, i.e., $\tilde\sigma = 4$,
employ three preconditioners:
the mean-based $P_0$ and the modified truncation preconditioners
$\tilde P_1$ and $\tilde P_2$.
\abrev{Note that other modified
truncation preconditioners (for $r>2$) yield similar performance and
the corresponding results are not included here. We chose to work with two values of $M \in
\{4,\,8\}$ and several uniform subdivisions into squares of side
lengths ranging from $h=2^{-3}$ to $h=2^{-7}$, 
while keeping fixed the polynomial degree $k = 3$.}
%then we compute another sequence of SGFEM approximations
%now keeping fixed $h = 2^{-4}$ and varying $k = 1,\,2,\,3,\,4,\,5,\,6$.
%
% \todo{
% (the experiments
% %in the case of slow decay ($\tilde\sigma = 2$) as well as for 
% with `vanilla' truncation preconditioners $P_r$ ($r = 1,\ldots,6$) and
% with other modified truncation preconditioners $\tilde P_r$ ($r=3,\ldots,6$)
% show similar patterns and lead to the same conclusions about optimality with respect to $h$ and $M$??).
\abrev{The results of these computations are presented in Table~\ref{tab:ex5:hM-optimality}.
These results show that indeed,
the iteration counts do not change as $M$ increases from 4 to 8
(for the same value of $h$) and
as the spatial mesh gets sufficiently refined (for the same value of~$M$).
}

%----------------------------
\begin{table}[!t]
\begin{center}
{\small
	\begin{tabular}{c | *3{c} | *3{c} || *3{c} | *3{c}}
				  & \multicolumn{6}{c||}{fast decay} & \multicolumn{6}{c}{slow decay} \\
	\cline{2-13}
	\multirow{2}{*}{$h$} & \multicolumn{3}{c|}{{\large\strut} $M=4$} & \multicolumn{3}{c||}{$M=8$}
	                                & \multicolumn{3}{c|}{$M=4$} & \multicolumn{3}{c}{$M=8$} \\
	\cline{2-13}
	{\Large\strut} & $P_{0}$ & $\tilde P_{1}$ & $\tilde P_{2}$ & $P_{0}$ & $\tilde P_{1}$ & $\tilde P_{2}$
			  & $P_{0}$ & $\tilde P_{1}$ & $\tilde P_{2}$ & $P_{0}$ & $\tilde P_{1}$ & $\tilde P_{2}$\\
	\hline
		{\large \strut}\!\!
		$2^{-3}$ & 18 & 8 & 8		   &   18 & 8 & 8 	& 13 & 7 & 6		   &   13 & 7 & 6 \\
		$2^{-4}$ & 21 & 9 & 9   	   & 21 & 9 & 9	& 14 & 8 & 7   	   & 14 & 8 & 7\\
		$2^{-5}$ & 23 & 10 & 9 	    & 23 & 10 & 9	& 14 & 8 & 7 	   & 15 & 8 & 7\\
		$2^{-6}$ & 24 & 10 & 10	   & 24 & 10 & 10	& 15 & 8 & 7	 	   & 15 & 8 & 7\\
		$2^{-7}$ & 24 & 10 & 10	   & 24 & 10 & 10	& 15 & 8 & 7	 	   & 15 & 8 & 7\\
	\end{tabular}
	\caption{PCG iterations counts for test problem~1 % expansion coefficients \Refx{eq:ex5}
	for various $h$ and $M$, for $k=3$.}
	\label{tab:ex5:hM-optimality}
}
\end{center}
\end{table}
%----------------------------

% section 5.2
%----------------------------
\subsection{Test problem 2: non-affine parametric diffusion coefficient} \label{sec:numerics:non-affine}

Consider again model problem \Refx{eq:strong:form}, now
with the following truncated \emph{lognormal} diffusion coefficient
\begin{equation}
  a(\x,\y) := \exp\!\big(b(\x,\y)\big) = \exp\!\bigg(b_0(\x) + \sum_{m=1}^N
    b_m(\x)y_m\bigg),\ \ \x \in \Omega,\ \ \y \in \bGamma := \prod_{m=1}^N\Gamma_m,
    \label{eq:ex5log}
\end{equation}
where $b_0,\,b_m \in L^{\infty}(\Omega)$ for all $m=1,\ldots,N$ and 
%Unlike the previous test case,
the parameters $y_m \in \Gamma_m := \RR$
are the images of independent normally distributed random variables with zero mean and unit variance.
Accordingly, $p_m$ now denotes the standard Gaussian probability density function,
and the joint probability density function is $p(\y)=\prod_{m=1}^N p_m(y_m)$.
The well-posedness of weak formulation~\Refx{eq:weak:form} in this case has been studied in~\cite{Charrier_12_SWE}.

As a polynomial basis of $L^2_\pi(\bGamma)$ we choose the set of scaled Hermite polynomials
$\{\psi_{\balpha}:\bGamma\rightarrow\RR : \balpha\in\NN_0^N\}$
orthonormal with respect to the inner product $\langle\cdot,\cdot\rangle_{\pi}$.
In this basis, the diffusion coefficient~\Refx{eq:ex5log} has the representation
%\begin{equation} \label{eq:gPC}
\[
  a(\x,\y):=\sum_{{\balpha \in \NN_0^N}} a_{\balpha}(\x) \psi_{\balpha}(\y)
\]
%\end{equation}
with (cf.~\cite[p.~926]{Ullmann10})
\begin{equation}
   a_\balpha(\x) = \langle a(\x,\cdot), \psi_{\balpha}\rangle_{\pi} = %& = %\int_{\bGamma} a(\x,\y) \psi_{\balpha}(\y) p(\y) \dd\y\quad
%   ({\balpha \in \NN_0^N}).
%\end{equation*}
%For this choice, the representation \Refx{eq:gPC} holds with (cf. \cite[p.~926]{Ullmann10})
%\begin{equation}
%   a_{\balpha}(\x) =
%   \int_{\bGamma} \exp\!\big(b(\x,\y)\big) \psi_{\balpha}(\y) p(\y) \dd\y
%   \nonumber
%   \\[5pt]
%   & =
%   \langle \exp(b), \psi_{\balpha}\rangle_{\pi} =
   \mathbb{E}[a(\x,\cdot)] \prod_{m\in\spp\balpha}\frac{b_{m}^{\alpha_{m}}(\boldsymbol{x})}{\sqrt{\alpha_{m}}}
   \quad
   \forall\,\balpha \in \NN_0^N,
   \label{eq:t-alpha}
\end{equation}
where
\[
   \mathbb{E}[a(\x,\cdot)] =
   \int_{\bGamma} \exp\sep{b(\x,\y)} p(\y) \dd\y =
   \exp\bigg(b_{0}(\x)+\frac{1}{2}\sum_{m=1}^N b_{m}^{2}(\x)\bigg) > 0\ \ \text{a.e. in $\Omega$}.
\]

Let $M<N$.
%Substituting the expansion~\Refx{eq:gPC} of the diffusion coefficient in the
%associated bilinear form (cf.~\Refx{eq:A-form}) and using
%defined in~\Refx{eq:VhkM}, we obtain the following equations
%for finding the coefficients $u_{ij}$ in the representation~\Refx{eq:approx:sol} of the Galerkin solution $\uhkM \in \VhkM$
%for $s=1,\ldots,N_{\x}$ and $t=1,\ldots,N_{\y}$ (cf.~\Refx{eq:sgfem1}):
A Galerkin projection onto $X_h\otimes S_{k}^M$ (cf. \Refx{eq:skm}--\Refx{eq:VhkM}) yields a linear system with coefficient~matrix
\begin{equation}
   A=\sum_{\balpha\in \I_{2k}^{M}} G_{\balpha} \otimes K_{\balpha},
   \label{eq:decomp:A:lognormal}
\end{equation}
where for all $\balpha \in \I_{2k}^{M}$ we have defined
\begin{equation*}
\begin{aligned}
	\left[G_{\balpha}\right]_{tj} & :=
	\left\langle \psi_{\balpha} \psi_{\bkappa(j)}, \psi_{\bkappa(t)}\right\rangle _{\pi}
	\qquad\ \,\text{for } t,j=1,\ldots,N_\y,\\[4pt]
	\left[K_{\balpha}\right]_{si} & :=
	\int_{\Omega} a_{\balpha} \nabla\phi_{i} \cdot \nabla\phi_{s}\, \dd\x
	\qquad \text{for } s,i=1,\ldots,N_\x
\end{aligned}
\end{equation*}
with $a_\balpha$ given by~\Refx{eq:t-alpha}. % and $N_\y = \card{\I_{k}^M}$.
Note that all matrices $G_\balpha$ have nonnegative entries; cf.~\cite[Appendix~A]{ErnstU_10_SGM}.
In particular, $G_\zr = I$, where $\zr = (0,0,\ldots,0) \in \I_{2k}^{M}$.
Furthermore, the well-posedness of weak formulation implies that the matrix $A$ is positive definite.

In order to define a family of truncation preconditioners for $A$, we introduce an
ordering of the terms in \Refx{eq:decomp:A:lognormal} based on the magnitude of $a_\balpha$.
To that end, we assume, without loss of generality, that all magnitudes
$\norm{a_\balpha}{\infty}$ for $\balpha \in \I_{2k}^M$ are distinct.
Then, the ordering is defined by the sequence
$\seq{\balpha_\ell\, :\, \ell = 0,1,\ldots,\card{\I_{2k}^M}-1}$
% $\seq{\balpha_\ell}_{1 \leq \ell \leq \card{\I_{2k}^M}}$
of all multi-indices in $\I_{2k}^M$ such that
\[
  \norm{a_{\balpha_i}}{\infty}>\norm{a_{\balpha_j}}{\infty},\quad i < j.
\]
%
%We note that in our experiments the largest magnitudes of 
%$a_\balpha$ were distinct and the above ordering was well-defined.
%
Unlike in the affine case, this ordering is not sufficient to ensure
positivity of the truncated diffusion coefficient
\begin{equation}
  a_r(\x,\y) := %a_\zr(\x)+\sum_{m=1}^ra_{\balpha_m}(\x)\psi_{\balpha_m}(\y).
  \sum_{\ell=0}^r a_{\balpha_\ell}(\x)\psi_{\balpha_\ell}(\y),\quad 0 \le r \le \card{\I_{2k}^M} - 1.
  \label{eq:lognormal:truncation}
\end{equation}
Consequently, the resulting truncation preconditioner
\begin{equation}
   P_r := %G_\zr\otimes K_\zr+\sum_{m=1}^{r}G_{\balpha_m}\otimes K_{\balpha_m}
   \sum_{\ell=0}^{r}G_{\balpha_\ell}\otimes K_{\balpha_\ell}
  \label{eq:lognormal:Pr}
\end{equation}
is not guaranteed to be positive definite.
%
%We therefore adjust the form of $P_r$ to
%\[
%P_{r,\varepsilon}:=G_\zr\otimes K_\zr+\varepsilon\sum_{m=1}^{r}G_{\balpha_m}\otimes K_{\balpha_m},~~~\varepsilon\in[0,1].
%\]
%We note that if $\varepsilon$ is sufficiently small, $P_{r,\varepsilon}$ is positive definite and thus a suitable
%preconditioner to use with the Conjugate Gradient method.
%
However, as in the affine case, we replace $P_{r}$ % and $P_{r,\epsilon}$
by its symmetric block Gauss--Seidel approximation $\tilde{P}_{r}$. %and $\tilde{P}_{r,\varepsilon}$.
As demonstrated in Proposition~\ref{prop:lognormal:SBGS} below, the modified truncation preconditioner $\tilde{P}_{r}$ is positive definite,
provided that the mean field $a_\zr(\x) = \mathbb{E}[a(\x,\cdot)]$ is included in the
truncation~\Refx{eq:lognormal:truncation}.
It is important to note that the complexity associated with the action of $\tilde{P}_{r}^{-1}$
% computational cost per iteration
remains unchanged from the affine case; cf.~\S\ref{sec:implement}.

\begin{proposition} \label{prop:lognormal:SBGS}
Let $a$ be the lognormal diffusion coefficient given by~\Refx{eq:ex5log}.
Let $0 \le r \le \card{\I_{2k}^M} - 1$ and assume that the truncated diffusion coefficient $a_r$ in~\Refx{eq:lognormal:truncation}
satisfies $\mathbb{E}\left[a_r\right] = \mathbb{E}\left[a\right] = a_\zr > 0$ a.e. in $\Omega$.
Then the truncation preconditioner $P_r$ given by~\Refx{eq:lognormal:Pr} can be represented as
\begin{equation}
   P_r = D + L + L^T,
   \label{eq:lognormal:Pr:decomp}
\end{equation}
where $D$ is a block-diagonal symmetric positive definite matrix and
$L$ is a strictly lower block-triangular matrix.
As a consequence, the SBGS approximation of $P_r$ defined~by
\[
   \tilde P_r := (D + L)\, D^{-1} (D + L^T)
\]
is a symmetric positive definite matrix.
\end{proposition}

\begin{proof}
Denote by $\I_{\rm \,even}$ the set of multi-indices in $\I_{2k}^M$ with even entries, i.e.,
\[
   \I_{\rm \,even} := 
   \left\{
   \balpha = (\alpha_1,\ldots,\alpha_M) \in \I_{2k}^M;\; \alpha_m\ \text{is even for all}\ m = 1,\ldots,M
   \right\}.
\]
Let $\balpha \in \I_{\rm \,even}$.
It follows from~\Refx{eq:t-alpha} that $a_\balpha(\x) \ge 0$ a.e. in $\Omega$.
Therefore, the stiffness matrices $K_\balpha$ for $\balpha \in \I_{\rm \,even}$ are symmetric positive semi-definite;
in particular, the matrix $K_\zr$ is symmetric positive definite.
Since all diagonal elements of the matrix $G_\balpha$ are nonnegative, 
each nonzero diagonal block of $G_\balpha \otimes K_\balpha$ is a symmetric positive semi-definite matrix.
In particular, the block-diagonal matrix $G_\zr \otimes K_\zr = I \otimes K_\zr$ is symmetric positive definite.

Now let $\balpha \in \I_{2k}^M \setminus \I_{\rm \,even}$. %and %$\beta \in \I_{2k}^M$.
In this case, all diagonal elements of the matrix $G_\balpha$ are zeros.
Indeed, for any $j = 1,\ldots,N_\y$, one has
\[
   \left[G_{\balpha}\right]_{jj} =
   \big\langle \psi_{\balpha} \psi_{\bkappa(j)}, \psi_{\bkappa(j)}\big\rangle _{\pi} =
   \big\langle \psi_{\balpha}, \psi^2_{\bkappa(j)}\big\rangle _{\pi} = 0,
\]
because there exists $m^* \in \{1,\ldots,M\}$ such that $\alpha_{m^*}$ is odd and the associated univariate Hermite polynomial
is an odd function.
Therefore, in this case, all diagonal blocks of $G_\balpha \otimes K_\balpha$ are zero matrices.

Overall, by combining the above observations %for diagonal blocks of 
%$G_\balpha \otimes K_\balpha$ with $\balpha \in \I_{\rm \,even}$ and $\balpha \in \I_{2k}^M \setminus \I_{\rm \,even}$,
and using the assumption that $\mathbb{E}\left[a_r\right] = \mathbb{E}\left[a\right]$,
we conclude that the diagonal blocks of the truncation preconditioner $P_r$ in~\Refx{eq:lognormal:Pr}
% (as, in fact, the diagonal blocks of the Galerkin matrix $A$ in~\Refx{eq:decomp:A:lognormal})
are symmetric positive definite matrices. This proves~\Refx{eq:lognormal:Pr:decomp}.

It is now easy to see that the SBGS approximation of $P_r$ is a positive definite matrix.
Indeed, for any nonzero vector $\bv$ there holds
\[
   \bv^T \tilde P_r \bv = \bv^T (D + L)\, D^{-1} (D + L^T) \bv = \w^T D^{-1} \w > 0
\]
with nonzero $\w := (D + L^T) \bv$.
\end{proof}

\begin{table}[!t]
%	\begin{centering}
%		\par\end{centering}
	\centering{}%
	\begin{tabular}{c|c|c}
		$\ell$ & $\balpha_\ell$ & $\norm{a_{\balpha_\ell}}{\infty}$\tabularnewline
		\hline 
		0 & (0,0,0,0,0,0) & 3.20\tabularnewline
		1 & (1,0,0,0,0,0) & 1.75\tabularnewline
		2 & (2,0,0,0,0,0) & 0.68\tabularnewline
		3 & (0,1,0,0,0,0) & 0.44\tabularnewline
		4 & (1,1,0,0,0,0) & 0.24\tabularnewline
		5 & (3,0,0,0,0,0) & 0.21\tabularnewline
		6 & (0,0,1,0,0,0) & 0.19\tabularnewline
		7 & (0,0,0,1,0,0) & 0.11
	\end{tabular}
	\caption{Multi-indices of first 8 largest magnitudes $\norm{a_{\balpha}}{\infty}$ for test problem 2; $M = k = 6$.}
	\label{tab:magnitude:non-affine}
\end{table}

In numerical experiments, we set $N = 20$ and chose $b_m(\x)$ in~\Refx{eq:ex5log}
to be the coefficients $a_m(\x)$ in test problem~1 as defined in~\Refx{eq:ex5}
with $\tilde{\sigma}=2$ and $\bar{\alpha}=0.547$.
In Table~\ref{tab:magnitude:non-affine}, for $M = k = 6$, we show first eight multi-indices in the sequence
% $\seq{\balpha_\ell}_{1 \leq \ell \leq \card{\I_{2k}^M}}$
$\seq{\balpha_\ell}$
and the corresponding coefficient magnitudes $\norm{a_{\balpha_\ell}}{\infty}$.
We see that in this example, the coefficient with the largest magnitude is the mean field, i.e., $a_{\balpha_0} = a_\zr$.
% corresponded to the multi-index $\balpha_1:=\zr$,
%i.e., $\norm{a_{\balpha_1}}{L^\infty(\Omega)} = \norm{a_\zr}{L^\infty(\Omega)} > \norm{a_{\balpha_\ell}}{L^\infty(\Omega)}$ for $\ell>1$,
%although this may not be the case in general.
%
While the distribution of indices inducing the ordering does not display any obvious pattern, %. On the other hand,
we note a fast decay with $\ell$ in the magnitudes recorded, which is similar to the affine~case.

Table~\ref{tab:numeric:non-affine} displays the PCG iteration counts corresponding to solving linear systems arising from SGFEM
discretizations of the described test problem.
%\Refx{eq:strong:form} with diffusion coefficient \Refx{eq:t-alpha}, with $N=20$ and where $b_{0}$ and $b_{m}$ are defined as
%in \Refx{eq:ex5} with $\tilde{\sigma}=2$
%and $\bar{\alpha}=0.547$.
We used the following discretization parameters: $h \,{=}\, 2^{-4}$, $M \,{=}\, 6$, and $k \,{\in}\, \{1,\ldots,6\}$.
In our experiments, we employed the modified truncation preconditioners $\tilde P_r$ with $r \in \{1,\ldots,6\}$,
%(i.e., the SBGS approximations $P_r$ in~\Refx{eq:lognormal:Pr}),
alongside the mean-based ($P_0$) and Kronecker ($P_\otimes$) preconditioners.
These experiments included cases where preconditioners $P_{r}$ defined by~\Refx{eq:lognormal:Pr} were not positive definite
(in Table~\ref{tab:numeric:non-affine}, % the iteration counts corresponding to such cases are given in~boldface).
the iteration counts for such cases are shown in~boldface).
%
% The results show the robustness of $\tilde P_r$ to such positivity-deficient truncations.
%
% In particular, such cases arise with increasing $k$. When this was observed, $\tilde{P}_{r}$ was replaced with $\tilde{P}_{r,\varepsilon}$, with
%\[
%  \varepsilon:=0.9\sep{\rho\left(G_{\balpha_2}\right)\rho\left(K_{\zr}^{-1}K_{\balpha_2}\right)}^{-1},
%\]
%where $\rho$ denotes the spectral radius. This value of $\varepsilon$ was derived under
%the assumption of fast decay in the magnitudes of $a_\balpha$, with
%the relaxation parameter 0.9 sufficient in our experiments to ensure positivity of $\tilde{P}_{r,\varepsilon}$.

\begin{table}[!b]
	\centering{}%
	\begin{tabular}{r|cccccccc}
		 $k$ & $P_{\otimes}$ & $P_{0}$ & $\tilde{P}_{1}$ & $\tilde{P}_{2}$ & $\tilde{P}_{3}$ & $\tilde{P}_{4}$ & $\tilde{P}_{5}$ & $\tilde{P}_{6}$\\
		\hline 
		1 & 12 & 12 & 6 & 7 & 6 & 6 & 6 & 6\\[-1pt]
		2 & 18 & 19 & 8 & 10 & 9 & 9 & 8 & 8\\[-1pt]
		3 & 25 & 26 & \bf{10} & 12 & 11 & 11 & 10 & 10\\[-1pt]
		4 & 32 & 34 & \bf{13} & 15 & 13 & 13 & 12 & 11\\[-1pt]
		5 & 40 & 43 & \bf{17} & 19 & 16 & 17 & \bf{13} & \bf{12}\\[-1pt]
		6 & 49 & 52 & \bf{24} & 22 & 19 & 20 & \bf{14} & \bf{14}
	\end{tabular}
	\caption{PCG iterations counts %for lognormal diffusion coefficient; $h=2^{-4}$, $M=6$.}
	    for test problem 2; $h=2^{-4}$, $M=6$.}
	\label{tab:numeric:non-affine}
\end{table}

The results in Table \ref{tab:numeric:non-affine} indicate that the numbers of iterations by the modified truncation preconditioners
are significantly lower than those corresponding to the mean-based and Kronecker preconditioners
(it is worth noting here that while the computational cost for $P_0$ and $\tilde P_r$ remains unchanged from the affine case,
the cost for $P_{\otimes}$ %the Kronecker preconditioner
in this test problem will be significantly higher than in the affine case,
due to the density of the matrix $G$ in~\Refx{eq:kron} for the lognormal diffusion coefficient).
%Moreover,
For all preconditioners, the experiments show that the iteration counts grow with $k$,
although this growth is much less pronounced for % our
truncation preconditioners.
Furthermore, while we see only a negligible improvement with increasing $r$ for $k = 1,\ldots,4$,
this becomes more pronounced for higher polynomial degrees ($k = 5,6$).
%although the case $r=1$ requires the
%estimation of the parameter $\varepsilon$ for $k\geq 3$.

%------------section6------------------------------------------------
\section{Summary and future work} \label{sec:conclusions}

Efficient solution of large coupled linear systems is a key ingredient in successful implementation
of the stochastic Galerkin finite element method.
\rev{Truncation preconditioners %\rev{studied} in this work
represent} a competitive alternative to
existing solvers relying on the mean-based and Kronecker preconditioners.
Our theoretical analysis shows that for elliptic problems with \emph{affine-parametric} coefficients,
% for problems with affine representations of the diffusion coefficient 
truncation preconditioners are optimal with respect to discretization parameters.
Our numerical experiments confirm this, while also demonstrating the improvement in the iteration count
when compared with the mean-based and Kronecker preconditioners.

On a practical note, the superior efficiency of \rev{considered} solvers requires,
crucially, suitable fast (possibly parallel) implementation of the corresponding symmetric block Gauss--Seidel approximations,
which were also analyzed and shown to be optimal.
For simplicity, we considered a model diffusion problem, however, the analysis included in this work can be
extended in a straightforward manner to the general case of elliptic PDE with parametric or uncertain inputs, under standard assumptions.

We have also % considered an extension of % our preconditioning strategy
\rev{applied truncation preconditioners in} % to
the case of \emph{non-affine} (specifically, lognormal) diffusion coefficient.
The numerical experiments suggest this is a promising approach.
Theoretical analysis of truncation preconditioners for this class of parametric problems will be the focus of future research on the~topic.

%-------------------------------------------------------------------
%\newpage
\bibliographystyle{siam}
\bibliography{ref}
%-------------------------------------------------------------------

\end{document}